\documentclass[preprint,12pt]{elsarticle}
\usepackage[utf8]{inputenc}
\usepackage[british]{babel}
\usepackage{adjustbox}
\usepackage{amsmath,amssymb,amsthm}
\usepackage{mathtools}
\usepackage{enumitem}
\usepackage{float}
\usepackage{multirow}
\usepackage{tikz}
\usepackage{accents}
\usepackage{hyperref}
\usepackage[nameinlink,capitalise,noabbrev]{cleveref}

\newlist{assump}{enumerate}{1}
\setlist[assump,1]{label=(A\arabic*),ref=(A\arabic*)}
\setlist[assump]{resume}
\crefname{assumpi}{assumption}{assumptions}

\newtheorem{lemma}{Lemma}
\newtheorem{problem}[lemma]{Problem}
\newtheorem{theorem}[lemma]{Theorem}

\theoremstyle{definition}

\newtheorem{remark}[lemma]{Remark}

\def\div{\operatorname{div}}

\def\RR{\mathbb{R}}

\def\LL{\mathbf{L}}
\def\dt{\partial_t}
\def\dphi{\partial_\phi}
\def\dgradphi{\partial_{\nabla\phi}}
\def\dtheta{\partial_\vtheta}
\def\dtau{d^{n+1}_\tau}

\def\M{\mathbf{M}}
\def\K{\mathbf{K}}
\def\C{\mathbf{C}}
\def\L{\mathbf{L}}
\def\I{\mathbf{I}}
\def\G{\mathbf{G}}
\def\x{\mathbf{x}}
\def\y{\mathbf{y}}
\def\b{\mathbf{b}}

\def\Itau{\mathcal{I}_\tau}
\def\Itauk{\mathcal{I}_{\tau_k}}

\def\la{\langle}
\def\ra{\rangle}
\def\vtheta{\vartheta}
\def\Th{\mathcal{T}_h}
\def\Vh{\mathcal{V}_h}

\def\div{\operatorname{div}}

\DeclarePairedDelimiter{\norm}{\|}{\|}
\DeclarePairedDelimiter{\snorm}{|}{|}

\def\softd{{\leavevmode\setbox1=\hbox{d}%
		\hbox to 1.05\wd1{d\kern-0.4ex{\char039}\hss}}}%cstocs

\journal{Applied Mathematics and Computation}

\begin{document}

\begin{frontmatter}

\title{Structure-preserving approximation of the non-isothermal Cahn-Hilliard system based on the entropy equation}
\author{Aaron Brunk\fnref{mainz}}
\cortext[cor1]{Corresponding author}
\ead{abrunk@uni-mainz.de}
\author{Dennis Höhn\corref{cor1}\fnref{mainz}}
\ead{dennis.hoehn@uni-mainz.de}
\author{M\'aria Luk\'a\v{c}ov\'a-Medvi\softd ov\'a\fnref{mainz}}
\ead{lukacova@uni-mainz.de}
\affiliation[mainz]{organization={Johannes Gutenberg-Universität},%Department and Organization
addressline={Staudingerweg 9}, 
city={Mainz},
postcode={55099}, 
country={Germany}}

\begin{abstract}
We present and investigate a structure-preserving approximation of the non-isothermal Cahn-Hilliard equation, employing conforming finite elements for spatial discretisation and a tailored mixed explicit-implicit scheme for time integration. To guarantee the preservation of key structural properties, namely mass and internal energy conservation, along with entropy production, we formulate the continuous problem within an appropriate variational framework based on the entropy equation. Our analytical results are validated through numerical experiments, including a convergence study

% Phase-field models are widely used to model phenomena in binary systems, especially for the temperature-dependent separation of the two phases.
% However, the simulation of such systems, especially with respect to the physical property's is no easy task, and the current discretisation techniques are not thermodynamically consistent.
% Here I present a structure-preserving approximation of the non-isothermal Cahn-Hilliard system which fulfills the laws of thermodynamics on a discrete level. This allows for accurate simulation of the above phenomena and even achieves better performance compared to fully implicit approaches because of relaxed requirements to the time step size.
\end{abstract}

\begin{keyword}
Finite elements \sep Phase-field \sep Structure-preserving Approximation\sep Non-isothermal \sep Cahn-Hilliard
\end{keyword}

\end{frontmatter}

\section{Introduction}\label{sec:intro}

The Cahn-Hilliard equation is a well-established model in physics for describing conservative phase transitions, such as magnetism \cite{Gunton1983} and spinodal decomposition \cite{CH,Bray01061994}, particularly in binary systems like alloys or liquid mixtures. These transitions typically occur when a system is quenched below a critical temperature, prompting phase separation. In many applications, the system is studied under isothermal conditions, assuming a constant temperature. However, the influence of nonisothermal effects has gained attention in recent work \cite{Lebedev2019,polym15163475,https://doi.org/10.1002/cjce.24873}. Cahn-Hilliard-type models have also been applied beyond materials science, including tumour growth modelling, both in isothermal \cite{Trauti24} and nonisothermal contexts \cite{IPOCOANA2022125665}. In materials science, non-isothermal phase-field models are particularly relevant for processes such as sintering and powder bed fusion \cite{Yang2019, yangscripta2020,Timi23}, where local temperature fluctuations can drive the system above or below critical thresholds. These scenarios necessitate non-isothermal variants of the Cahn-Hilliard equation. In this work, we study the non-isothermal Cahn-Hilliard (NCH) equation in the framework of Pawlow and Alt \cite{Alt1992}, given by the system:
\begin{align}
    \dt\phi - \div\left(\M\nabla\tfrac{\mu}{\vtheta}-\C\nabla\tfrac{1}{\vtheta}\right) &= 0,  \label{eq:sys1}\\
    \tfrac{\mu}{\vtheta} + \gamma\Delta\phi - \dphi \tfrac{f(\phi,\vtheta)}{\vtheta}&= 0, \label{eq:sys2}\\
    \dt e(\phi,\vtheta) - \div\left(\C\nabla\tfrac{\mu}{\vtheta}-\K\nabla\tfrac{1}{\vtheta}\right) &= 0. \label{eq:sys3}
\end{align}
Here $\phi$ is the phase-field variable distinguishing between phase A ($\phi=0$) and phase B($\phi=1$), $\tfrac{\mu}{\vtheta}$ is the chemical potential and $\vtheta$ is the absolute temperature. The matrices $\M=\M(\phi,\nabla\phi,\vtheta),\K=\K(\phi,\nabla\phi,\vtheta)$ and $\C=\C(\phi,\nabla\phi,\vtheta)$ represent the mobility, thermal conductivity, and cross-coupling coefficients and may depend on $\phi,\nabla\phi$ and $\vtheta$.
The system is posed on a periodic spatial domain $\Omega$ over a finite time interval $(0,T)$, supplemented by initial conditions and subject to the second law of thermodynamics, ensuring non-decreasing entropy over time. The Helmholtz free energy density $F(\phi,\nabla\phi,\vtheta)$ is given as follows
\begin{equation}
    F(\phi,\nabla\phi,\vtheta) := \frac{\gamma\vtheta}{2}\snorm{\nabla\phi}^2 + f(\phi,\vtheta),\label{eq:helmholtz}
\end{equation}
The relationship between the chemical potential $\frac{\mu}{\vtheta}$, internal energy density $e$ and entropy density $s$ is then expressed via
\begin{align}
    \frac{\mu}{\vtheta}&:=\delta_\phi \frac{F(\phi,\nabla\phi,\vtheta)}{\vtheta} =-\gamma\Delta\phi + \dphi \tfrac{f(\phi,\vtheta)}{\vtheta}\\
    s(\phi,\nabla\phi,\vtheta)&:=-\dtheta F(\phi,\nabla\phi,\vtheta) = -\frac{\gamma}{2}\snorm{\nabla\phi}^2 - \partial_\vtheta f(\phi,\vtheta),\\
    e(\phi,\vtheta)&:=F(\phi,\nabla\phi,\vtheta)+\vtheta s(\phi,\nabla\phi,\vtheta)= f(\phi,\vtheta) - \vtheta \partial_\vtheta f(\phi,\vtheta).
\end{align}
This structure guarantees the following thermodynamic properties:
\begin{gather*}
  \int_\Omega \phi(t) = \int_\Omega \phi(0) , \qquad \int_\Omega e(\phi(t),\vtheta(t))  = \int_\Omega e(\phi(0),\vtheta(0)) , \\
  \int_\Omega s(\phi(t),\nabla\phi(t),\vtheta(t)) = \int_\Omega s(\phi(0),\nabla\phi(0),\vtheta(0)) + \int_0^t \mathcal{D}\left(\tfrac{\mu}{\vtheta},\tfrac{1}{\vtheta}\right)
  \intertext{where the entropy production $\mathcal{D}$ is given by}
  \mathcal{D}\left(\tfrac{\mu}{\vtheta},\tfrac{1}{\vtheta}\right) = \int_\Omega\begin{pmatrix}
      \nabla \frac{1}{\vtheta} & \nabla \frac{\mu}{\vtheta}
  \end{pmatrix}\begin{pmatrix}
      \K & -\C \\ -\C & \M
  \end{pmatrix}\begin{pmatrix}
      \nabla \frac{1}{\vtheta} \\ \nabla \frac{\mu}{\vtheta}
  \end{pmatrix}.
\end{gather*}
If the Onsager matrix
\begin{equation}\label{eq:onsag}
    \L=\begin{pmatrix}
      \K & -\C \\ -\C & \M
    \end{pmatrix},
\end{equation}
which depends on $\phi,\nabla\phi$ and $\vtheta$, i.e. $\L=\L(\phi,\nabla\phi,\vtheta)$, is positive (semi)-definite, then the entropy $s(\phi,\nabla\phi,\vtheta)$ is nondecreasing in time.\\

%% Theory:
Various extensions of non-isothermal phase-field models have been developed in the literature, encompassing both conserved and non-conserved dynamics. General frameworks can be found in \cite{Pawlow2016,Heida2011}. The non-conserved case was notably addressed by Penrose and Fife \cite{Penrose1990}, while the conserved formulation of Alt and Pawłow \cite{Alt1992} uses entropy as the driving potential. These models can also be derived using the GENERIC or SNET formalism \cite{Gladkov2016}, or from a microforce balance perspective following Gurtin's framework \cite{GURTIN1996178,article,Marveggio2021}, or via variational principles such as the least action \cite{deanna2022temperature}. Finite-speed thermal propagation has also been addressed, e.g., in models by Caginalp \cite{Caginalp1986,COLLI2024113461}. Further extensions to include (quasi-)incompressible fluid flow are discussed in \cite{Alessia2014,Eleutri15}. Analytical studies on these models, including the existence of weak solutions, especially in the presence of cross-coupling terms, can be found in \cite{alt1992existence,KENMOCHI19941163,Colli_memory,Colli_Penrose,COLLI2024113461}.\\

%% Numeric:
On the numerical side, an implicit Euler method using inverse temperature as a variable was first introduced in \cite{Pawlow2016}, and later enhanced for energy stability by González et al. \cite{GonzalezFerreiro2014} via the Average Vector Field (AVF) method \cite{Gonzales96}. More recently, Brunk et al. \cite{BrunkPamm,BrunkCMAM} proposed structure-preserving schemes for such systems using the inverse temperature as the main variable. The Cahn-Hilliard and Allen-Cahn equations, known to exhibit a gradient structure, have been approximated using techniques like convex-concave splitting \cite{Jie}, the Scalar Auxiliary Variable (SAV) approach \cite{SHEN2018407,LI23}, and Energy Quadratisation (EQ) \cite{CHEN2022,Zhang2022}. These methods introduce auxiliary variables to relax the energy or entropy functional. In particular, EQ methods have been applied to the internal energy equation in \cite{Guo_2015,Sun2020}.\\

The goal of this work is to investigate the variational structure of the NCH system \eqref{eq:sys1}–\eqref{eq:sys3} and to develop a systematic discretisation in space and time using standard temperature variable, instead of the inverse temperature. The main contributions of this paper are:
\begin{itemize}
\item A reformulation of the system using the entropy equation in place of the internal energy equation to highlight the underlying variational structure.
\item A structure-preserving numerical approximation, employing conventional discretisation techniques in both space and time.
%\item Nonlinear stability estimate for the semi-discrete and the fully discrete scheme by means of relative entropy.
\end{itemize}

The manuscript is organised as follows:
In Section~\ref{sec:Prel}, we introduce notation and reformulate the system by replacing the internal energy equation with its entropy-based counterpart. In Section~\ref{sec:Var}, we derive a variational formulation of the reformulated system suitable for standard finite element approximation. Section~\ref{sec:fulldiscrete} and Section~\ref{sec:structurepresurving} present the fully discrete scheme based on a mixed explicit-implicit time integration and convex-concave splitting of the potential. Numerical results, including convergence studies and an application-motivated example, are provided in Section~\ref{sec:numerical}.

\section{Preliminaries}\label{sec:Prel}
Before delving into the details of our discretisation strategy and main results, we first introduce the necessary notation, state our primary assumptions, and recall some basic definitions.

%%%%%%%%%%%%%%%%%%%%%%%%%%%%%%%%%%%%%%%%%%
\subsection{Notation}
We study the system \eqref{eq:sys1}--\eqref{eq:sys3} over a finite time interval $(0,T).$ To simplify the analysis and avoid complications related to boundary conditions, we assume a periodic spatial domain:
\begin{assump}
    \item Let $\Omega \subset \RR^d$, $d=1,2,3$ be a cube, identified with the $d$-dimensional torus $\mathcal{T}^d$. That is, we impose periodic boundary conditions and assume that the functions on $\Omega$ are periodic throughout the paper. 
\end{assump}
We denote by $L^p(\Omega)$ and $W^{k,p}(\Omega)$ the standard periodic Lebesgue and Sobolev spaces, equipped with norms $\norm{\cdot}_{L^{p}}$ and $\norm{\cdot}_{W^{k,p}}$, respectively. 
In particular, we set $H^k(\Omega)=W^{k,2}(\Omega)$ and use $\norm{\cdot}_{H^{k}} = \norm{\cdot}_{W^{k,2}}$. 
%
% The corresponding dual spaces are denoted by $H^{-k}(\Omega)=H^k(\Omega)^*$, with norm defined by
% \begin{align*} \label{eq:dualnorm}
%     \norm{r}_{H^{-k}} = \sup_{v \in H^s(\Omega)} \frac{\la r, v\ra}{\|v\|_{k}}.
% \end{align*}
%
%
The duality pairing on $H^{-k}(\Omega) \times H^k(\Omega)$ is denoted by $\langle \cdot, \cdot\rangle$.
We use the same notation for scalar product on $L^2(\Omega)$, which is defined as
\begin{align*}
\la u, v \ra = \int_\Omega u \cdot v \, dx \qquad \forall u,v \in L^2(\Omega).    
\end{align*}
%
%
%

% As usual, we denote by $L^p(a,b;X)$, $W^{k,p}(a,b;X)$, and
% $H^k(a,b;X)$, the Bochner spaces of integrable or differentiable functions on the time interval $(a,b)$ with values in some Banach space $X$. If $(a,b)=(0,T)$, we omit reference to the time interval and briefly write $L^p(X)$. The corresponding norms are denoted, e.g., by $\|\cdot\|_{L^p(X)}$ or $\|\cdot\|_{H^k(X)}$.

For later use, we recall the following topological degree theorem from \cite{gallouet2008unconditionally} which will be crucial in proving the existence of discrete solutions in our analysis:
\begin{theorem}\label{th:topo}
Let N and M be two positive integers and $C_1, C_2$ and $\varepsilon$ three positive constants. We define 
\begin{align*}
    V &= \{(\x, \y) \in \RR^N\times\RR^M ~\text{such that}~\y > 0\},\\
    W &= \{(\x, \y)\in\RR^N\times\RR^M~\text{s.t.}~\|\x\|<C_1 ~\text{and}~ \varepsilon<\y<C_2\}.
\end{align*}
Here, the notation $\y>c$ means that each component of $\y$ is greater than a constant $c\in\RR$. Further $\|\cdot\|$ is a norm defined over $\RR^N$. Let
$\b\in\RR^N\times\RR^M$, $\mathbf{g}$ and $\G$ be two continuous functions, respectively, from $V$ and $V\times[0, 1]$ to $\RR^N\times\RR^M$
satisfying the following conditions:
\begin{enumerate}[label=(\roman*)]
    \item $\G(\cdot, 1) = \mathbf{g}(\cdot)$
    \item The topological degree of $\G(\cdot, 0):=\G_0(\cdot)$ with respect to $W$ and $\b$ is equal to $d_0\neq0$, i.e.
    \begin{equation*}
        \mathrm{deg}(\G_0,W,\b):= \sum_{\x\in \G_0^{-1}(\b)}\mathrm{sgn}\left(\det\left(J_{\G_0}(\x)\right)\right) = d_0.
    \end{equation*}
    \item For all $\alpha\in[0, 1]$, if $\mathbf{v}\in V$ is such that $\G(\mathbf{v},\alpha) = \b$ then $\mathbf{v}\in W$.
\end{enumerate}
Then the topological degree of $\G(\cdot, 1)$ with respect to $W$ and $\b$ is also equal to $d_0\neq0$. Consequently, there exists
at least one solution $\mathbf{v}\in W$, such that $\mathbf{g}(\mathbf{v}) = \b$.
\end{theorem}

\subsection{Assumptions}\label{sec:assump}

In this subsection, we summarize the minimal set of assumptions that will be used throughout the remainder of the manuscript. 

\begin{assump}
    \item The interface parameter $\gamma\in\RR$ is strictly positive.
    \item The Onsager matrix $\LL=\LL(\phi,\nabla\phi,\vtheta)\in\mathbb{R}^{2d \times 2d}$ from \eqref{eq:onsag} is positive definite. Specifically, there exist constants $\lambda_0,\lambda_1>0$ such that for all $\mathbf{\xi}\in\mathbb{R}^{2d}$
    \begin{equation*}
       \lambda_0\snorm{\xi}^2 \leq \mathbf{\xi}^\top\LL\mathbf{\xi} \leq \lambda_1\snorm{\xi}^2.
    \end{equation*}\label{as:onsag}
    \item The driving potential $f(\cdot,\cdot):\RR\times\RR_+\to \RR$ is smooth. For each fixed $\phi\in\mathbb{R}$, the function  $f(\phi,\cdot):\RR_+\to\RR$ is concave and diverges as $\vtheta\to 0$. Moreover, for each fixed $\vtheta >0$, the function $f(\cdot,\vtheta)$ can be decomposed into a convex and a concave part denoted by $f_{\text{vex}}$ and $ f_{\text{cav}}$, such that $f(\cdot,\vtheta)=f_{\text{vex}}(\cdot,\vtheta)+f_{\text{cav}}(\cdot,\vtheta)$.\label{as:pot}

\end{assump}

\subsection{Structural identities}

On the formal level one can choose between several formulations of the temperature contribution which allows to switch between the internal energy $e$ and the entropy $s$.

\begin{lemma}\label{th:spp}
    Regular solutions $(\phi,\mu,\vtheta)$ to the non-isothermal Cahn-Hilliard system \eqref{eq:sys1}-\eqref{eq:sys3} satisfy
    \begin{gather}
        \la \dt\phi,1 \ra = \la \dt e(\phi,\vtheta),1 \ra = 0,
        \quad\la \dt s(\phi,\nabla\phi,\vtheta),1 \ra = %\la \nabla\tfrac{1}{\vtheta},\K\nabla\tfrac{1}{\vtheta} \ra - 2\la\nabla\tfrac{1}{\vtheta},\C\nabla\tfrac{\mu}{\vtheta}\ra + \la\nabla\tfrac{\mu}{\vtheta}, \M\nabla\tfrac{\mu}{\vtheta} \ra =
        \mathcal{D}\left(\tfrac{\mu}{\vtheta},\tfrac{1}{\vtheta}\right) \geq 0.
    \end{gather}  
\end{lemma}
Before considering the proof of \cref{th:spp} we first derive the entropy equation. Therefore we recall the identity $e=F+\vtheta s$ which after rearrangement yields $s=(e-F)/\vtheta$ and
\begin{align*}
 \dt s(\phi,\nabla\phi,\vtheta) &= \dt \tfrac{e(\phi,\vtheta)-F(\phi,\nabla\phi,\vtheta))}{\vtheta}\notag\\
 & = \tfrac{1}{\vtheta}\dt e(\phi,\vtheta) - \tfrac{1}{\vtheta}\dt F(\phi,\nabla\phi,\vtheta) - \tfrac{e(\phi,\vtheta)-F(\phi,\nabla\phi,\vtheta)}{\vtheta^2}\dt \vtheta \notag\\
 & = \tfrac{1}{\vtheta}\dt e(\phi,\vtheta) - \tfrac{1}{\vtheta}\dphi F(\phi,\nabla\phi,\vtheta)\dt\phi - \tfrac{1}{\vtheta}\dgradphi F(\phi,\nabla\phi,\vtheta)\dt\nabla\phi \notag\\
 &\quad - \tfrac{1}{\vtheta}\dtheta F(\phi,\nabla\phi,\vtheta)\dt\vtheta - \tfrac{s(\phi,\nabla\phi,\vtheta)}{\vtheta}\dt \vtheta. \notag\\
 \intertext{Using the relation $s(\phi,\nabla\phi,\vtheta)=-\dtheta F(\phi,\nabla\phi,\vtheta)$ and substituting in \eqref{eq:helmholtz} leads to}
 \dt s(\phi,\nabla\phi,\vtheta) &= \tfrac{1}{\vtheta}\dt e(\phi,\vtheta) - \tfrac{1}{\vtheta}\dphi f(\phi,\vtheta)\dt\phi - \gamma\nabla\phi\cdot\nabla\dt\phi. \notag\\ 
 \intertext{Insertion of the internal energy equation \eqref{eq:sys3} and the definition of the chemical potential $\tfrac{\mu}{\vtheta}$, cf. \eqref{eq:sys2} then yields}
 \dt s(\phi,\nabla\phi,\vtheta) &= \tfrac{1}{\vtheta}\div(\C\nabla\tfrac{\mu}{\vtheta}-\K\nabla\tfrac{1}{\vtheta}) - \tfrac{\mu}{\vtheta}\dt\phi - \gamma\div(\dt\phi\nabla\phi). \notag\\
 \intertext{Now inserting the Cahn-Hilliard equation cf. \eqref{eq:sys1} implies}
 \dt s(\phi,\nabla\phi,\vtheta) &= \tfrac{1}{\vtheta}\div(\C\nabla\tfrac{\mu}{\vtheta}-\K\nabla\tfrac{1}{\vtheta}) - \tfrac{\mu}{\vtheta}\div(\M\nabla\tfrac{\mu}{\vtheta}-\C\nabla\tfrac{1}{\vtheta})\notag\\
 &\quad- \gamma\div(\dt\phi\nabla\phi). %\\
\end{align*}
By using the reverse product rule we finally obtain the entropy equation
\begin{align}\label{eq:sys3alternative}
    \dt s(\phi,\nabla\phi,\vtheta) &= \nabla\tfrac{1}{\vtheta}\cdot\K\nabla\tfrac{1}{\vtheta} - 2\nabla\tfrac{1}{\vtheta}\cdot\C\nabla\tfrac{\mu}{\vtheta} + \nabla\tfrac{\mu}{\vtheta}\cdot\M\nabla\tfrac{\mu}{\vtheta} \notag\\
    &- \div\left(\gamma\dt\phi\nabla\phi + \tfrac{1}{\vtheta}\K\nabla\tfrac{1}{\vtheta} - \tfrac{1}{\vtheta}\C\nabla\tfrac{\mu}{\vtheta} - \tfrac{\mu}{\vtheta}\C\nabla\tfrac{1}{\vtheta} + \tfrac{\mu}{\vtheta}\M\nabla\tfrac{\mu}{\vtheta} \right).
\end{align}
% where
% \begin{equation}
%     \mathbf{L}=\left(\begin{array}{cc}
%         \K & -\C \\
%         -\C & \M
%     \end{array}\right)\quad\mathbf{r}=\left(\begin{array}{c}
%         \nabla\tfrac{1}{\vtheta}\\
%         \nabla\tfrac{\mu}{\vtheta}
%     \end{array}\right)
% \end{equation}

\begin{proof}[Proof of \cref{th:spp}]
Conservation of mass and internal energy as well as entropy production simply follow from the respective equations \eqref{eq:sys1},\eqref{eq:sys3} and \eqref{eq:sys3alternative} by integration over the domain and partial integration.
\end{proof}

Note that to obtain the entropy production for system \eqref{eq:sys1}--\eqref{eq:sys3} we need to test with $\tfrac{1}{\vtheta}$ which is not directly amendable for discretizations via finite elements. However, using the entropy equation instead of the internal energy equation implies that we only need to test with $\vtheta$ to get the energy conservation, while  entropy production follows by simple integration over the domain and using partial integration for the flux.

System \eqref{eq:sys1},\eqref{eq:sys2},\eqref{eq:sys3alternative} is equivalent to \eqref{eq:sys1}--\eqref{eq:sys3} and therefore also satisfies conservation of mass and internal energy as well as entropy production.

\section{Variational structure}\label{sec:Var}

In the this section we will reformulate system \eqref{eq:sys1}, \eqref{eq:sys2}, \eqref{eq:sys3alternative} into a variational formulation suitably such that conservation of mass and energy as well as entropy production follows directly by insertion of suitable test functions. 

\begin{lemma}
    Regular solutions $(\phi,\mu,\vtheta)$ to system \eqref{eq:sys1}, \eqref{eq:sys2}, \eqref{eq:sys3alternative} satisfy the variational formulation 
    \begin{align}
        \la\dt\phi,\psi\ra&=\la\M\tfrac{\mu}{\vtheta^2}\nabla \vtheta-\tfrac{\M}{\vtheta}\nabla\mu-\tfrac{\C}{\vtheta^2}\nabla\vtheta,\nabla\psi\ra\label{eq:var1}\\
        \la\mu,\xi\ra&=\gamma\la \nabla\phi,\vtheta\nabla\xi+\xi\nabla \vtheta\ra+\la\dphi f,\xi\ra\label{eq:var2}\\
        \la\dt s,\omega\ra&=\la\tfrac{\K}{\vtheta^2}\nabla \vtheta+\tfrac{\C}{\vtheta}\nabla\mu-2\C\tfrac{\mu}{\vtheta^2}\nabla \vtheta-\M\tfrac{\mu}{\vtheta}\nabla\mu+\M\left(\tfrac{\mu}{\vtheta}\right)^2\nabla \vtheta,\tfrac{\omega}{\vtheta^2}\nabla \vtheta-\tfrac{1}{\vtheta}\nabla\omega\ra \notag\\
        &\quad+\la\tfrac{\C}{\vtheta^2}\nabla\vtheta+\tfrac{\M}{\vtheta}\nabla\mu-\M\tfrac{\mu}{\vtheta^2}\nabla \vtheta,\tfrac{\omega}{\vtheta}\nabla\mu\ra +\gamma\la\dt\phi\nabla\phi,\nabla\omega\ra
        \label{eq:var3}
    \end{align}
    for smooth test functions $\psi,\xi,\omega.$
\end{lemma}
\begin{proof}
First we multiply the equations \eqref{eq:sys1}, \eqref{eq:sys2} and \eqref{eq:sys3alternative} by the functions $\psi,\vtheta\xi$ and $\omega$, respectively, and integrate over the domain $\Omega$ to obtain
\begin{align*}
    \la\dt\phi,\psi\ra&=\la\div(\M\nabla\tfrac{\mu}{\vtheta}-\C\nabla\tfrac{1}{\vtheta}),\psi\ra\\
    \la\mu,\xi\ra&=-\gamma\la\Delta\phi,\vtheta\xi\ra+\la\dphi f,\xi\ra\\
    \la\dt s,\omega\ra&=\la\nabla\tfrac{1}{\vtheta}\cdot\K\nabla\tfrac{1}{\vtheta} - 2\nabla\tfrac{1}{\vtheta}\cdot\C\nabla\tfrac{\mu}{\vtheta} + \nabla\tfrac{\mu}{\vtheta}\cdot\M\nabla\tfrac{\mu}{\vtheta},\omega\ra \notag\\
    &\quad- \la\div\left(\gamma\dt\phi\nabla\phi + \tfrac{1}{\vtheta}\K\nabla\tfrac{1}{\vtheta} - \tfrac{1}{\vtheta}\C\nabla\tfrac{\mu}{\vtheta} - \tfrac{\mu}{\vtheta}\C\nabla\tfrac{1}{\vtheta} + \tfrac{\mu}{\vtheta}\M\nabla\tfrac{\mu}{\vtheta} \right),\omega\ra.
    %\la\partial_ts,\omega\ra&=\la\tfrac{1}{\vtheta}\div(\C\nabla\tfrac{\mu}{\vtheta}-\K\nabla\tfrac{1}{\vtheta}),\omega\ra - \la\tfrac{\mu}{\vtheta}\div(\M\nabla\tfrac{\mu}{\vtheta}-\C\nabla\tfrac{1}{\vtheta}),\omega\ra\\
    %&\quad- \gamma\la\div(\dt\phi\nabla\phi),\omega\ra.
\end{align*}
In the first equation we use partial integration and derive
\begin{align*}
    \la\partial_t\phi,\psi\ra&=-\la \M\nabla\tfrac{\mu}{\vtheta}-\C\nabla\tfrac{1}{\vtheta},\nabla\psi\ra.
    % \intertext{Applying the chain rule to $\nabla\tfrac{\mu}{\vtheta}$ and $\nabla\tfrac{1}{\vtheta}$ we immediately get}
    % \la\partial_t\phi,\psi\ra&=\la\tfrac{\M}{\vtheta^2}\mu\nabla \vtheta -\tfrac{\M}{\vtheta}\nabla\mu-\tfrac{\C}{\vtheta^2}\nabla\vtheta,\nabla\psi\ra.
\end{align*}
Applying the chain rule to $\nabla\tfrac{\mu}{\vtheta}$ and $\nabla\tfrac{1}{\vtheta}$ we immediately get \eqref{eq:var1}. 
In the second equation we use partial integration
\begin{align*}
    \la\mu,\xi\ra&=\gamma\la\nabla\phi,\nabla(\xi \vtheta)\ra+\la\dphi f,\xi\ra
    % \intertext{and apply the product rule to get}
    % \la\mu,\xi\ra&=\gamma\la \nabla\phi,\vtheta\nabla\xi+\xi\nabla \vtheta\ra+\la\dphi f,\xi\ra.
\end{align*}
and apply the product rule to get \eqref{eq:var2}.
Finally, in the third equation we again use partial integration
\begin{align*}
     \la\partial_ts,\omega\ra&=\la\nabla\tfrac{1}{\vtheta}\cdot\K\nabla\tfrac{1}{\vtheta} - 2\nabla\tfrac{1}{\vtheta}\cdot\C\nabla\tfrac{\mu}{\vtheta} + \nabla\tfrac{\mu}{\vtheta}\cdot\M\nabla\tfrac{\mu}{\vtheta},\omega\ra \notag\\
    &\quad+ \la\gamma\dt\phi\nabla\phi + \tfrac{1}{\vtheta}\K\nabla\tfrac{1}{\vtheta} - \tfrac{1}{\vtheta}\C\nabla\tfrac{\mu}{\vtheta} - \tfrac{\mu}{\vtheta}\C\nabla\tfrac{1}{\vtheta} + \tfrac{\mu}{\vtheta}\M\nabla\tfrac{\mu}{\vtheta},\nabla\omega\ra.
    % \la\partial_ts,\omega\ra&=-\la \C\nabla\tfrac{\mu}{\vtheta}-\K\nabla\tfrac{1}{\vtheta},\nabla\tfrac{\omega}{\vtheta}\ra + \la \M\nabla\tfrac{\mu}{\vtheta}-\C\nabla\tfrac{1}{\vtheta},\nabla\tfrac{\mu\omega}{\vtheta}\ra\\
    % &\quad+ \gamma\la\dt\phi\nabla\phi,\nabla\omega\ra.
    \intertext{Combining with the chain rule for the fraction terms we get}
    \la\dt s,\omega\ra&=\la\tfrac{\K}{\vtheta^2}\nabla\vtheta + 2\tfrac{\C}{\vtheta}\nabla\mu-2\C\tfrac{\mu}{\vtheta^2}\nabla\vtheta - 2\M\tfrac{\mu}{\vtheta}\nabla\mu + \M\left(\tfrac{\mu}{\vtheta}\right)^2\nabla\vtheta,\tfrac{\omega}{\vtheta^2}\nabla\vtheta\ra\\ &\quad+ \la\tfrac{\M}{\vtheta}\nabla\mu,\tfrac{\omega}{\vtheta}\nabla\mu\ra + \la\gamma\dt\phi\nabla\phi,\nabla\omega\ra\notag\\
    &\quad+ \la-\tfrac{\K}{\vtheta^2}\nabla\vtheta - \tfrac{\C}{\vtheta}\nabla\mu + 2\C\tfrac{\mu}{\vtheta^2}\nabla\vtheta + \M\tfrac{\mu}{\vtheta}\nabla\mu - \M\left(\tfrac{\mu}{\vtheta}\right)^2\nabla\vtheta,\tfrac{1}{\vtheta}\nabla\omega\ra,
    % \la\dt s,\omega\ra&=\la\tfrac{\K}{\vtheta^2}\nabla \vtheta+\tfrac{\C}{\vtheta}\nabla\mu-\C\tfrac{\mu}{\vtheta^2}\nabla \vtheta,\tfrac{\omega}{\vtheta^2}\nabla\vtheta-\tfrac{1}{\vtheta}\nabla\omega\ra\\
    % &\quad+\la\tfrac{\M}{\vtheta}\nabla\mu-\M\tfrac{\mu}{\vtheta^2}\nabla \vtheta+\tfrac{\C}{\vtheta^2}\nabla\vtheta,\tfrac{\omega}{\vtheta}\nabla\mu+\tfrac{\mu}{\vtheta}\nabla\omega-\tfrac{\mu\omega}{\vtheta^2}\nabla \vtheta\ra\\
    % &\quad+\gamma\la\dt\phi\nabla\phi,\nabla\omega\ra,
\end{align*}
which after rearrangement yields \eqref{eq:var3}.
\end{proof}
\begin{theorem}\label{th:varspp}
    Regular solutions $(\phi,\mu,\vtheta)$ to the variational formulation \eqref{eq:var1}--\eqref{eq:var3} satisfy
    \begin{align*}
        \la \dt\phi,1 \ra = \la \dt e,1 \ra = 0, \quad \la \dt s,1 \ra =\mathcal{D}\left(\tfrac{\mu}{\vtheta},\tfrac{1}{\vtheta}\right)\geq0.
  \end{align*}  
  Moreover, these properties follow directly by using test functions $\psi\in\{1,\mu\}, \xi=\dt\phi, w\in\{1,\vtheta\}.$
\end{theorem}
\begin{proof}
    By setting $\psi=1$ we immediately find 
    \begin{equation*}
        \la \dt\phi,1 \ra =\la\M\tfrac{\mu}{\vtheta^2}\nabla \vtheta-\tfrac{\M}{\vtheta}\nabla\mu-\tfrac{\C}{\vtheta^2}\nabla\vtheta,\nabla 1\ra=0.
    \end{equation*}
    The entropy production follows by using the test function $\omega=1$:
    \begin{align*}
        \la\dt s,1\ra&=\la\tfrac{\K}{\vtheta^2}\nabla \vtheta+\tfrac{\C}{\vtheta}\nabla\mu-2\C\tfrac{\mu}{\vtheta^2}\nabla\vtheta-\M\tfrac{\mu}{\vtheta}\nabla\mu+\M\left(\tfrac{\mu}{\vtheta}\right)^2\nabla \vtheta,\tfrac{1}{\vtheta^2}\nabla \vtheta\ra\\
        &\quad+\la\tfrac{\C}{\vtheta^2}\nabla\vtheta+\tfrac{\M}{\vtheta}\nabla\mu-\M\tfrac{\mu}{\vtheta^2}\nabla \vtheta,\tfrac{1}{\vtheta}\nabla\mu\ra\\
        &=\la\K\tfrac{1}{\vtheta^2}\nabla\vtheta,\tfrac{1}{\vtheta^2}\nabla\vtheta\ra - 2\la\C\left(\tfrac{\mu}{\vtheta^2}\nabla\vtheta-\tfrac{1}{\vtheta}\nabla\mu\right),\tfrac{1}{\vtheta^2}\nabla\vtheta\ra\\
        &\quad+ \la\M\left(\tfrac{\mu}{\vtheta^2}\nabla\vtheta-\tfrac{1}{\vtheta}\nabla\mu\right),\tfrac{\mu}{\vtheta^2}\nabla\vtheta-\tfrac{1}{\vtheta}\nabla\mu\ra\\
        &=\la\K\nabla\tfrac{1}{\vtheta},\nabla\tfrac{1}{\vtheta}\ra - 2\la\C\nabla\tfrac{\mu}{\vtheta},\nabla\tfrac{1}{\vtheta}\ra + \la\M\nabla\tfrac{\mu}{\vtheta},\nabla\tfrac{\mu}{\vtheta}\ra\\
        &=\mathcal{D}\left(\tfrac{\mu}{\vtheta},\tfrac{1}{\vtheta}\right)\geq0.
    \end{align*}
    Using the identities $e=F+\vtheta s$ and $s=-\dtheta F$, energy conservation follows:
    \begin{align*}
        \la\dt e,1\ra&=\la\dt(F+\vtheta s),1\ra\\
        &=\la\dphi F\dt\phi+\dtheta F\dt \vtheta+\dt \vtheta s+\vtheta \dt s,1\ra\\
        &=\la\gamma \vtheta\nabla\phi\cdot\nabla\dt\phi+\dphi f \dt\phi+\vtheta \dt s,1\ra\\
        &=\gamma\la \nabla\phi,\vtheta\nabla\dt\phi\ra+\la\dphi f,\dt\phi\ra+\la\dt s,\vtheta\ra.\\
        \intertext{Taking $\xi=\dt\phi$ in \eqref{eq:var2} and $\omega=\vtheta$ in \eqref{eq:var3} we further get}
        \la\dt e,1\ra&=\la\mu,\dt\phi\ra-\gamma\la\nabla\phi,\dt\phi\nabla \vtheta\ra+\gamma\la\dt\phi\nabla\phi,\nabla \vtheta\ra\\
        &\quad+\la\tfrac{\C}{\vtheta^2}\nabla\vtheta+\tfrac{\M}{\vtheta}\nabla\mu-\M\tfrac{\mu}{\vtheta^2}\nabla \vtheta,\nabla\mu\ra \\
        \intertext{and using \eqref{eq:var1} with $\psi=\mu$ we finally obtain}
        \la\dt e,1\ra&=\la\M\tfrac{\mu}{\vtheta^2}\nabla \vtheta-\tfrac{\M}{\vtheta}\nabla\mu-\tfrac{\C}{\vtheta^2}\nabla\vtheta,\nabla\mu\ra\\
        &\quad+\la\tfrac{\C}{\vtheta^2}\nabla\vtheta+\tfrac{\M}{\vtheta}\nabla\mu-\M\tfrac{\mu}{\vtheta^2}\nabla \vtheta,\nabla\mu\ra=0.
    \end{align*}
\end{proof}

\section{Fully discrete scheme for the non-isothermal Cahn-Hilliard equation}\label{sec:fulldiscrete}
As preparatory step, we introduce the notation and assumptions that will be used in our discretisation strategy. 

\subsection*{Time discretization}
We consider a uniform partition of the time interval $[0,T]$ with time step size $\tau>0$ and define the discrete time grid $\Itau:=\{0=t^0,t^1=\tau,\ldots, t^{n_\tau}=T\}$, where $n_\tau=\tfrac{T}{\tau}$ denotes the absolute number of time steps. We introduce $\Pi^1_c(\Itau;X), \Pi^0(\Itau;X)$ as the spaces of continuous piecewise linear and piecewise constant functions on $\Itau$ with values in a suitable Hilbert space $X$, respectively. For a function $g\in \Pi^1_c(\Itau;X),\Pi^0(\Itau;X)$, we use the notation $g^{n+1},g^n,g^*$ to refer to its evaluation at $t^{n+1},t^n$, and at some intermediate time $t^*\in[t^n,t^{n+1}]$, respectively. The time difference and the backward finite difference quotient (discrete time derivative) are defined by
\begin{equation*}
	d^{n+1}g = g^{n+1} - g^n, \qquad d^{n+1}_\tau g = \frac{g^{n+1}-g^n}{\tau},
\end{equation*}
respectively.

\subsection*{Space discretization}

For the spatial discretization, we consider a geometrically conforming triangulation $\Th$ of the domain $\Omega$ into simplices with maximal diameter $h.$ The mesh is assumed to be periodically extendable to match the periodic structure of $\Omega$. 
We define the space of continuous, piecewise linear finite element functions over $\Th$ by
\begin{align*}
	\Vh &:= \{v \in H^1(\Omega)\cap C^0(\bar\Omega) : v|_K \in P_1(K),~\forall K \in \Th\}.
\end{align*}

For $\phi_h,\vtheta_h\in\Pi^1_c(\Itau;\Vh)$ and $n\in\{0,...,n_\tau-1\}$ we adopt a convex-concave splitting with respect to the phase variable $\phi$ of the potential $f$, and define the discrete approximation of its derivative as
\begin{equation*}
	f_{\phi}(\phi_h^{n+1},\phi_h^n,\vtheta_h^{n+1}) := \partial_\phi f_{\text{vex}}(\phi_h^{n+1},\vtheta_h^{n+1}) + \partial_\phi 
 f_{\text{cav}}(\phi_h^{n},\vtheta_h^{n+1}).   
\end{equation*}
For the components of the Onsager matrix $\mathbf{X}\in\{\M,\K,\C\}$ we denote by $\mathbf{X}_h^*$ the evaluation at $(\phi_h^*,\nabla\phi_h^*,\vtheta_h^*)$ for some intermediate time $t^*\in[t^n,t^{n+1}]$, i.e. $\mathbf{X}_h^*=\mathbf{X}(\phi_h^*,\nabla\phi_h^*,\vtheta_h^*)$.

We now formulate a fully discrete time-stepping method for the non-isothermal Cahn–Hilliard system \eqref{eq:sys1}, \eqref{eq:sys2}, and \eqref{eq:sys3alternative}.

\begin{problem}\label{prob:fulldisc}
	Let the initial data $(\phi_{h}^0,\vtheta_{h}^0)\in \Vh\times\Vh$ be given. Find $(\phi_h,\vtheta_h)\in \Pi^1_c(\Itau;\Vh\times\Vh)$ and $\mu_{h}\in \Pi^0(\Itau;\Vh)$ such that
	% \begin{align}
 %        \la\dtau\phi_h,\psi_h\ra&=\la\M_h^*\tfrac{\mu_h^*}{\vtheta_h^{n+1}\cdot\vtheta_h^*}\nabla\vtheta_h^{n+1}-\tfrac{\M_h^*}{\vtheta_h^{n+1}}\nabla\mu_h^{n+1}-\tfrac{\C_h^*}{\vtheta_h^{n+1}\cdot\vtheta_h^*}\nabla\vtheta_h^{n+1},\nabla\psi_h\ra\label{eq:dis1}\\
 %        \la\mu_h^{n+1},\xi_h\ra&=\gamma\la\nabla\phi_h^{n+1},\vtheta_h^{n+1}\nabla\xi_h\ra+\gamma\la\nabla\phi_h^*,\xi_h\nabla \vtheta_h^{n+1}\ra\\
 %        &\quad+\la f_{\phi}(\phi_h^{n+1},\phi_h^n,\vtheta_h^{n+1}),\xi_h\ra\label{eq:dis2}\\
 %        \la\dtau s_h,\omega_h\ra&=\la\tfrac{K_h^*}{\vtheta_h^{n+1}\cdot\vtheta_h^*}\nabla \vtheta_h^{n+1}+\tfrac{\C_h^*}{\vtheta_h^{n+1}}\nabla\mu_h^{n+1}-2\C_h^*\tfrac{\mu_h^*}{\vtheta_h^{n+1}\cdot\vtheta_h^*}\nabla\vtheta_h^{n+1}\notag\\
 %        &\quad-\M_h^*\tfrac{\mu_h^*}{\vtheta_h^{n+1}}\nabla\mu_h^{n+1}+\M_h^*\tfrac{\left(\mu_h^*\right)^2}{\vtheta_h^{n+1}\cdot\vtheta_h^*}\nabla\vtheta_h^{n+1},\tfrac{\omega_h}{\vtheta_h^{n+1}\cdot\vtheta_h^*}\nabla\vtheta_h^{n+1}-\tfrac{1}{\vtheta_h^*}\nabla\omega_h\ra \notag\\
 %        &\quad+\la\tfrac{\C_h^*}{\vtheta_h^{n+1}\cdot\vtheta_h^*}\nabla\vtheta_h^{n+1}+\tfrac{\M_h^*}{\vtheta_h^{n+1}}\nabla\mu_h^{n+1}-\M_h^*\tfrac{\mu_h^*}{\vtheta_h^{n+1}\cdot\vtheta_h^*}\nabla\vtheta_h^{n+1},\tfrac{\omega_h}{\vtheta_h^{n+1}}\nabla\mu_h^{n+1}\ra\notag\\
 %        &\quad+\gamma\la\dtau\phi_h\nabla\phi_h^*,\nabla\omega_h\ra\label{eq:dis3}
	% \end{align}
    \begin{align}
        \la\dtau\phi_h,\psi_h\ra&=\la\tfrac{\mu_h^*}{\vtheta_h^*}\M_h^*\nabla\vtheta_h^{n+1}-\M_h^*\nabla\mu_h^{n+1}-\tfrac{1}{\vtheta_h^*}\C_h^*\nabla\vtheta_h^{n+1},\tfrac{1}{\vtheta_h^{n+1}}\nabla\psi_h\ra\label{eq:dis1}\\
        \la\mu_h^{n+1},\xi_h\ra&=\gamma\la\nabla\phi_h^{n+1},\vtheta_h^{n+1}\nabla\xi_h\ra+\gamma\la\nabla\phi_h^*,\xi_h\nabla \vtheta_h^{n+1}\ra \notag\\
        &\quad+\la f_{\phi}(\phi_h^{n+1},\phi_h^n,\vtheta_h^{n+1}),\xi_h\ra\label{eq:dis2}\\
        \la\dtau s_h,\omega_h\ra&=\la\tfrac{1}{\vtheta_h^*}\K_h^*\nabla \vtheta_h^{n+1}+\C_h^*\nabla\mu_h^{n+1}-2\tfrac{\mu_h^*}{\vtheta_h^*}\C_h^*\nabla\vtheta_h^{n+1}-\mu_h^*\M_h^*\nabla\mu_h^{n+1}\notag\\
        &\quad+\tfrac{(\mu_h^*)^2}{\vtheta_h^*}\M_h^*\nabla\vtheta_h^{n+1},\tfrac{\omega_h}{(\vtheta_h^{n+1})^2\vtheta_h^*}\nabla\vtheta_h^{n+1}-\tfrac{1}{\vtheta_h^*\vtheta_h^{n+1}}\nabla\omega_h\ra \notag\\
        &\quad+\la\tfrac{1}{\vtheta_h^*}\C_h^*\nabla\vtheta_h^{n+1}+\M_h^*\nabla\mu_h^{n+1}-\tfrac{\mu_h^*}{\vtheta_h^*}\M_h^*\nabla\vtheta_h^{n+1},\tfrac{\omega_h}{(\vtheta_h^{n+1})^2}\nabla\mu_h^{n+1}\ra\notag\\
        &\quad+\gamma\la\nabla\phi_h^*,\dtau\phi_h\nabla\omega_h\ra\label{eq:dis3}
	\end{align}
	holds for $(\psi_h,\xi_h,\omega_h)\in\Vh\times\Vh\times\Vh$ and all $0\leq n\leq n_T-1$.
\end{problem}

\section{Structure-preserving property}\label{sec:structurepresurving}

In this section, we establish the existence of at least one solution to \cref{prob:fulldisc}, and we demonstrate that the numerical method \eqref{eq:dis1}–\eqref{eq:dis3} preserves key structural properties at the discrete level—analogous to those described in \cref{th:varspp}. In particular, we verify discrete analogues of mass conservation, energy dissipation, and entropy production.

\begin{theorem}\label{th:disspp}
	Let $(\phi_h,\mu_{h},\vtheta_h)$ be any discrete solution of \cref{prob:fulldisc} with $\vtheta_h>0.$ Then the scheme preserves the following structural properties:
	\begin{align*}
\langle d^{n+1}_\tau \phi_h, 1 \rangle &= 0, \qquad & &\text{(discrete mass conservation)} \\
\langle d^{n+1}_\tau e_h, 1 \rangle &= \mathcal{D}^{n+1}_{\mathrm{num}} \leq 0, \qquad & &\text{(discrete energy dissipation)} \\
\langle d^{n+1}_\tau s_h, 1 \rangle &= \mathcal{D}_h^{n+1}\geq 0, \qquad & &\text{(discrete entropy production)}
\end{align*}
where $e_h$ and $s_h$ denote the discrete internal energy and entropy, respectively. The discrete entropy production reads
    \begin{equation*}
        \mathcal{D}_h^{n+1} =\int_\Omega \begin{pmatrix}
            \tfrac{\nabla\vtheta_h^{n+1}}{\vtheta_h^{n+1}\cdot\vtheta_h^*} &
             \tfrac{\mu_h^*}{\vtheta_h^{n+1}\cdot\vtheta_h^*}\nabla\vtheta_h^{n+1}-\tfrac{\nabla\mu_h^{n+1}}{\vtheta_h^{n+1}}
        \end{pmatrix}\begin{pmatrix}
            \K_h^* & -\C_h^* \\
            -\C_h^* & \M_h^*
        \end{pmatrix}\begin{pmatrix}
            \tfrac{\nabla\vtheta_h^{n+1}}{\vtheta_h^{n+1}\cdot\vtheta_h^*} \\
             \tfrac{\mu_h^*}{\vtheta_h^{n+1}\cdot\vtheta_h^*}\nabla\vtheta_h^{n+1}-\tfrac{\nabla\mu_h^{n+1}}{\vtheta_h^{n+1}}
        \end{pmatrix}
        % \mathbf{L}_h^*=\left(\begin{array}{cc}
        %     \K_h^* & -\C_h^* \\
        %     -\C_h^* & \M_h^*
        % \end{array}\right)\quad\mathbf{r}_h=\left(\begin{array}{c}
        %     \tfrac{\nabla\vtheta_h^{n+1}}{\vtheta_h^{n+1}\cdot\vtheta_h^*} \\
        %      \tfrac{\mu_h^*}{\vtheta_h^{n+1}\cdot\vtheta_h^*}\nabla\vtheta_h^{n+1}-\tfrac{\nabla\mu_h^{n+1}}{\vtheta_h^{n+1}}
        % \end{array}\right)
    \end{equation*}
    and the numerical dissipation is given as follows
    \begin{align*}
	 	\mathcal{D}_{num}^{n+1} = -\tfrac{\tau}{2}&\Big(\gamma\la\snorm{\nabla d_\tau^{n+1}\phi_h}^2,\vtheta_h^{n+1}\ra-\la\partial_{\vtheta\vtheta}f(\phi_h^{n},\zeta_1),(d_\tau^{n+1} \vtheta_h)^2\ra\\&+\la\partial_{\phi\phi}f_{\text{vex}}(\zeta_2,\vtheta_h^{{n+1}})-\partial_{\phi\phi}f_{\text{cav}}(\zeta_3,\vtheta_h^{{n+1}}),(d_\tau^{n+1}\phi_h)^2\ra\Big).
    \end{align*}
\end{theorem}
\begin{proof}
    To derive the conservation of mass and the entropy production, one can proceed exactly as in the continuous case by testing \eqref{eq:dis1} with $\psi_h=1$ and \eqref{eq:dis3} with $\omega_h=1$, respectively. To get the energy conservation we first use the identity $e=f-\vtheta\cdot\dtheta f$ and calculate as follows
    \begin{align*}
        \la\dtau e_h,1\ra&=\la\dtau f_h-\dtau(\vtheta_h\cdot\dtheta f_h),1\ra\\
        &=\tfrac{1}{\tau}\left(\la f_h^{n+1}-f_h^n,1\ra-\la \vtheta_h^{n+1}\cdot\dtheta f_h^{n+1}-\vtheta_h^n\cdot\dtheta f_h^n,1\ra\right)\\
        &=\tfrac{1}{\tau}\left(\la f_h^{n+1}-f_h^n,1\ra-\la\dtheta f_h^{n+1}-\dtheta f_h^{n}, \vtheta_h^{n+1}\ra\right.\\
        &\quad\left.-\la\dtheta f_h^{n},\vtheta_h^{n+1}-\vtheta_h^n\ra\right)\\
        &=\tfrac{1}{\tau}\left(\la f_h^{n+1}-f_h^n,1\ra+\la s_h^{n+1}-s_h^{n}, \vtheta_h^{n+1}\ra\right.\\
        &\quad\left.+\tfrac{\gamma}{2}\la\snorm{\nabla\phi_h^{n+1}}^2-\snorm{\nabla\phi_h^{n}}^2,\vtheta_h^{n+1}\ra-\la\dtheta f_h^{n},\vtheta_h^{n+1}-\vtheta_h^n\ra\right)\\
        &=\tfrac{1}{\tau}\la f_h^{n+1}-f_h^n,1\ra-\la\dtheta f_h^{n},\dtau \vtheta_h\ra+\la \dtau s_h, \vtheta_h^{n+1}\ra\\
        &\quad+\gamma\la\nabla\phi_h^{n+1},\vtheta_h^{n+1}\nabla\dtau\phi_h\ra-\tfrac{\gamma\tau}{2}\la\snorm{\nabla\dtau\phi_h}^2,\vtheta_h^{n+1}\ra.\\
        \intertext{Now we use \eqref{eq:dis3} with $\omega_h=\vtheta_h^{n+1}$, where the first term cancels out to get}
        \la\dtau e_h,1\ra&=\tfrac{1}{\tau}\la f_h^{n+1}-f_h^n,1\ra-\la\dtheta f_h^{n},\dtau \vtheta_h\ra+\gamma\la\nabla\phi_h^*,\dtau\phi_h\nabla \vtheta_h^{n+1}\ra\\
        &\quad+\la\tfrac{1}{\vtheta_h^*}\C_h^*\nabla\vtheta_h^{n+1}+\M_h^*\nabla\mu_h^{n+1}-\tfrac{\mu_h^*}{\vtheta_h^*}\M_h^*\nabla\vtheta_h^{n+1},\tfrac{1}{\vtheta_h^{n+1}}\nabla\mu_h^{n+1}\ra\\
        &\quad+\gamma\la\nabla\phi_h^{n+1},\vtheta_h^{n+1}\nabla\dtau\phi_h\ra-\tfrac{\gamma\tau}{2}\la\snorm{\nabla\dtau\phi_h}^2,\vtheta_h^{n+1}\ra.\\
        \intertext{Insertion of $\xi_h=\dtau\phi_h$ into \eqref{eq:dis2} shows}
        \la\dtau e_h,1\ra&=\tfrac{1}{\tau}\la f_h^{n+1}-f_h^n,1\ra-\la\dtheta f_h^{n},\dtau \vtheta_h\ra+\la\mu_h^{n+1},\dtau\phi_h\ra\\
        &\quad+\la\tfrac{1}{\vtheta_h^*}\C_h^*\nabla\vtheta_h^{n+1}+\M_h^*\nabla\mu_h^{n+1}-\tfrac{\mu_h^*}{\vtheta_h^*}\M_h^*\nabla\vtheta_h^{n+1},\tfrac{1}{\vtheta_h^{n+1}}\nabla\mu_h^{n+1}\ra\\
        &\quad-\la f_{\phi}(\phi_h^{n+1},\phi_h^n,\vtheta_h^{n+1}),\dtau\phi_h\ra-\tfrac{\gamma\tau}{2}\la\snorm{\nabla\dtau\phi_h}^2,\vtheta_h^{n+1}\ra.\\
        \intertext{Finally using \eqref{eq:dis1} with $\psi_h=\mu_h^{n+1}$ yields}
        \la\dtau e_h,1\ra&=\tfrac{1}{\tau}\la f_h^{n+1}-f_h^n,1\ra-\la\dtheta f_h^{n},\dtau \vtheta_h\ra\\
        &\quad-\la f_{\phi}(\phi_h^{n+1},\phi_h^n,\vtheta_h^{n+1}),\dtau\phi_h\ra-\tfrac{\gamma\tau}{2}\la\snorm{\nabla\dtau\phi_h}^2,\vtheta_h^{n+1}\ra.\\
        \intertext{Adding $\pm f(\phi_h^{n},\vtheta_h^{n+1})$ and using \ref{as:pot} allows us the rewrite as follows}
        \la\dtau e_h,1\ra&=\tfrac{1}{\tau}\la f_{\text{vex}}(\phi_h^{n+1},\vtheta_h^{n+1})-f_{\text{vex}}(\phi_h^{n},\vtheta_h^{n+1}),1\ra\\
        &\quad-\la\dphi f_{\text{vex}}(\phi_h^{n+1},\vtheta_h^{n+1}),\dtau\phi_h\ra\\
        &\quad+\tfrac{1}{\tau}\la f_{\text{cav}}(\phi_h^{n+1},\vtheta_h^{n+1})-f_{\text{cav}}(\phi_h^{n},\vtheta_h^{n+1}),1\ra\\
        &\quad-\la\dphi f_{\text{cav}}(\phi_h^{n},\vtheta_h^{n+1}),\dtau\phi_h\ra\\
        &\quad+\tfrac{1}{\tau}\la f(\phi_h^{n},\vtheta_h^{n+1})-f(\phi_h^{n},\vtheta_h^{n}),1\ra-\la\dtheta f(\phi_h^n,\vtheta_h^n),\dtau \vtheta_h\ra\\
        &\quad-\tfrac{\gamma\tau}{2}\la\snorm{\nabla\dtau\phi_h}^2,\vtheta_h^{n+1}\ra.\\
        \intertext{Taylor-expansion of $f$ in $\phi$ and $\vtheta$ finally reveals}
        \la\dtau e_h,1\ra&=-\tfrac{\tau}{2}\Big(\la\partial_{\phi\phi}f_{\text{vex}}(\zeta_2,\vtheta_h^{n+1})-\partial_{\phi\phi}f_{\text{cav}}(\zeta_3,\vtheta_h^{n+1}),(d_\tau^{n+1}\phi_h)^2\ra\\
        &\quad-\la\partial_{\vtheta\vtheta}f(\phi_h^{n},\zeta_1),(d_\tau^{n+1} \vtheta_h)^2\ra+\gamma\la\snorm{\nabla d_\tau^{n+1}\phi_h}^2,\vtheta_h^{n+1}\ra\Big)=\mathcal{D}_{num}^{n+1}.
    \end{align*}
    The negativity of $\mathcal{D}_{num}^{k}$ follows from the properties of $f$ in \ref{as:pot} and the positivity of $\vtheta_h$.
\end{proof}

\subsection{Existence of positive solutions}

In the next step, we will establish existence of discrete solutions with uniformly positive temperature. To this end we need to restrict the assumptions on the potential $f$.
\begin{assump}
    % \item We extend (A4) and assume that $f(\phi,\cdot):\RR_+\to\RR$ is strictly concave for fixed $\phi$. Furthermore let and $c_f\geq 0$ and $\phi,\vtheta:\RR^d\times\RR_+\rightarrow\RR$ with $\vtheta>\underaccent{\bar}{\vtheta}$ for some constant $\underaccent{\bar}{\vtheta}>0$. If there exists a constant $C_3>0$ such that \label{as:exergy}
    % \begin{equation*}
    %      \la e(\phi,\vtheta)-c_fs(\phi,\vtheta) , 1\ra \leq C_3.
    % \end{equation*}
    % then there exists a constant $C_4>0$ such that
    % \begin{align*}
    %     \norm{\nabla\phi}_{L^2}^2 + \norm{\vtheta}_{L^1(\Omega)} \leq C_4.
    % \end{align*}
    \item We extend \ref{as:pot} and assume that $f$ is of the form
    \begin{equation*}
        f(\phi,\vtheta) = -b\vtheta\log(\vtheta) + \vtheta f_1(\phi) + f_2(\phi)
    \end{equation*}
    with $f_2(\phi)\geq -c$ for some constants $b,c>0$ independent of $\phi$ and $\vtheta$. Furthermore, we assume that there exists another constant $c_f>0$ independent of $\phi$ such that
    \begin{equation*}
        f_2(\phi) + c_ff_1(\phi) \geq -c.
    \end{equation*}
    By construction $F(\phi,\cdot):\RR_+\to\RR$ is strictly concave.\label{as:exergy}
    \item We assume that the convex-concave decomposition of $f$ fulfills
    \begin{equation*}
        |f_\phi(\phi_h^{n+1},\phi_h^n,\vtheta_h^{n+1})| \leq C_f(|\vtheta_h^{n+1}|+1)(\snorm{\phi_h^{n+1}}^6 + \snorm{\phi_h^n}^6 + 1)
    \end{equation*}
    for some constant $C_f>0$.
    % \item We assume that $|f_\phi(\phi_h^{n+1},\phi_h^n,\vtheta_h^{n+1})| \leq C(\snorm{\phi_h^{n+1}}^{p_1} + \snorm{\phi_h^{n}}^{p_2} + \snorm{\vtheta_h^{n+1}}^k + 1)$, with $p_1,p_2 \leq 6$ and $k\in\mathbb{N}$.
    \label{as:fphi}
    \item The mesh $\Th$ is quasi-uniform.\label{as:mesh}
\end{assump}
Having quasi-uniform meshes $\Th$ we can apply the inverse inequalities, see \cite[Theorem 4.5.11]{BrennerScott},
\begin{align} \label{eq:inverse}
    \|v_h\|_{H^1} &\le C_\text{semi} h^{-1} \|v_h\|_{L^2}, \\
    \|v_h\|_{L^p} &\le C_\text{inv} h^{d/p-d/q} \|v_h\|_{L^q}, 
\end{align}
which hold for all discrete functions $v_h \in \Vh$ and $1 \le q \le p \le \infty$.

\begin{remark}
\Cref{as:exergy} is satisfied for
\begin{align*}
    f&=a\big(2\phi^4-4\phi^3+(\tfrac{\vtheta-\vtheta_c}{d}+3)\phi^2
        -(\tfrac{\vtheta-\vtheta_c}{d}+1)\phi+0.25(\tfrac{\vtheta-\vtheta_c}{d}+1)\big)\notag\\
        &\quad-b(\vtheta \log(\tfrac{\vtheta}{\vtheta_c})+\vtheta-\vtheta_c)
\end{align*}
for $c_f=\vtheta_c,b=b,c=\min\{\vtheta_c\log(\vtheta_c),0\}$ given that $a,b,d,\vtheta_c >0$, see also \cite{alt1992existence} for a similar assumption. 
%Note that the assumption can be weakened to only grantee that $f$ is striktly concave and that for every $\phi,\vtheta$ the internal energy $e$ is bounded from below by $c_fs-c_e$ for some constants $c_f,c_e$, see also \cite{alt1992existence} for a similar assumption.
\end{remark}

\begin{theorem}\label{th:existence}
    Let \ref{as:exergy}--\ref{as:mesh} hold. For given $h,\tau>0$ there exists at least one discrete solution $(\phi_h,\mu_{h},\vtheta_h)$ of \cref{prob:fulldisc} such that $\vtheta_h^k>\underaccent{\bar}{\vtheta}>0$ for all $0\leq k\leq n_T.$ 
\end{theorem}
\begin{proof}
    To prove the existence result it suffices to consider a single time step $0\leq n\rightarrow n+1\leq n_\tau$. For simplicity we consider $g^*=g^n+1$ and $\mathbf{X}_h^*=\mathbf{X}_h^n+1$ for all $g\in\{\phi_h,\mu_h\vtheta_h\}$ and $\mathbf{X}\in\{\M,\K,\C\}$ in \cref{prob:fulldisc}, but the whole proof works for other choices of $g^*$ and $\mathbf{X}_h^*$. After choosing the Lagrange basis in $\Vh$ we can rewrite \cref{prob:fulldisc} as a finite-dimensional nonlinear problem $\mathbf{g}(\x)=\mathbf{0}$, i.e. $\mathbf{b}=0$, where $\x=(\boldsymbol{\phi},\boldsymbol{\mu},\boldsymbol{\vtheta})\in\mathbb{R}^{3N}$ is a vector containing the nodal values of $(\phi^{n+1}_h, \mu_h^{n+1},\vtheta_h^{n+1})$. For simplicity we will use the abbreviation $(\phi_h,\mu_h,\vtheta_h)\equiv(\phi_h^{n+1},\mu_h^{n+1},\vtheta_h^{n+1})$ as well as $\mathbf{X}_h\equiv\mathbf{X}_h^{n+1}$. To apply \cref{th:topo} we consider
    \begin{align}
        V &:= \{(\boldsymbol{\phi},\boldsymbol{\mu},\boldsymbol{\vtheta})\in\mathbb{R}^{3N} \text{ such that } \boldsymbol{\vtheta}>0\}, \label{eq:defV}\\
        W &:=\{(\boldsymbol{\phi},\boldsymbol{\mu},\boldsymbol{\vtheta})\in\mathbb{R}^{3N} \text{ such that } \norm{\boldsymbol{\phi}} + \norm{\boldsymbol{\mu}}< C_{\phi,\mu}, \varepsilon <\boldsymbol{\vtheta}< C_\vtheta,\}. \label{eq:defW}
    \end{align}
    The constants $C_{\phi,\mu},C_\vtheta,\varepsilon$ will be chosen later.
    Note that for piecewise linear finite element functions point values in a triangle are convex combinations of the corresponding nodal values, i.e. $\boldsymbol{\vtheta}>0$ is equivalent to $\vtheta_h>0$ everywhere in $\Omega.$ We will frequently use the notation $(\phi_h,\mu_h,\vtheta_h)\in V$ or $W$, meaning that the vector $\x=(\boldsymbol{\phi},\boldsymbol{\mu},\boldsymbol{\vtheta})$ of the nodal values is in $V$ or $W$ respectively. In the following $(\phi_h^n,\mu_h^n,\vtheta_h^n)$ is assumed to be given and bounded. Moreover there exist positive constants $\underaccent{\bar}{\vtheta}^n, \bar\vtheta^{n}>0$ such that $\underaccent{\bar}{\vtheta}^n \leq \vtheta_h^n \leq \bar\vtheta^{n}$.%, i.e. $(\phi_h^n,\mu_h^n,\vtheta_h^n)\in W$.\\

    \noindent\textbf{Extended problem and condition (i):}\\
    First, we define $\G(\x,\alpha):=\la T_\alpha x_h,y_h \ra=\mathbf{0}$ for $x_h=(\phi_h,\mu_h,\vtheta_h)\in V$ and all basis functions $y_h=(\psi_h,\xi_h,\omega_h)\in\Vh\times\Vh\times\Vh$ with $\la T_\alpha x_h,y_h \ra=0$ given by the variational problem
    \begin{align}
        &\frac{1}{\tau}\la\phi_h-\phi_h^n,\psi_h\ra-\alpha\la\tfrac{\M_h}{\vtheta_h^2}\mu_h\nabla\vtheta_h-\tfrac{\M_h}{\vtheta_h}\nabla\mu_h-\tfrac{\C_h}{\vtheta_h^2}\nabla\vtheta_h,\nabla\psi_h\ra = 0\label{eq:exdis1}\\
        &\la\mu_h,\xi_h\ra-\gamma\la\nabla\phi_h,\vtheta_h\nabla\xi_h\ra-\gamma\la\nabla\phi_h,\xi_h\nabla \vtheta_h\ra-\la  f_{\phi}(\phi_h,\phi_h^n,\vtheta_h),\xi_h\ra=0\label{eq:exdis2}\\
        &\frac{1}{\tau}\la s(\phi_h,\vtheta_h)-s(\phi_h^n,\vtheta_h^n),\omega_h\ra-\alpha\la\tfrac{\K_h}{\vtheta_h^2}\nabla \vtheta_h+\tfrac{\C_h}{\vtheta_h}\nabla\mu_h-2\tfrac{\C_h}{\vtheta_h^2}\mu_h\nabla\vtheta_h\notag\\
        &\quad-\M_h\tfrac{\mu_h}{\vtheta_h}\nabla\mu_h+\M_h\tfrac{\mu_h}{\vtheta_h}^2\nabla\vtheta_h,\tfrac{\omega_h}{\vtheta_h^2}\nabla\vtheta_h-\tfrac{1}{\vtheta_h}\nabla\omega_h\ra \notag\\
        &\quad-\alpha\la\tfrac{\C_h}{\vtheta_h^2}\nabla\vtheta_h+\tfrac{\M_h}{\vtheta_h}\nabla\mu_h-\tfrac{\M_h}{\vtheta_h^2}\mu_h\nabla\vtheta_h,\tfrac{\omega_h}{\vtheta_h}\nabla\mu_h\ra\notag\\
        &\quad+\frac{1}{\tau}\gamma\la\nabla\phi_h(\phi_h-\phi_h^n),\nabla\omega_h\ra =0\label{eq:exdis3}
    \end{align}
    By construction it holds $\G(\cdot,1)=\mathbf{g}(\cdot)=0$, hence condition $(i)$ of \cref{th:topo} is fulfilled.\\

    \noindent\textbf{Condition (ii): Topological degree of $\G(\cdot,0)$:}\\
    In this part we show that the problem $\G(\x,0):=\G_0(\x)=0$ has a solution $\x\in W$ and that the topological degree of $\G_0$ is non-zero. We start by formulating $\G_0(\x)=0$, i.e. \eqref{eq:exdis1}--\eqref{eq:exdis3} with $\alpha=0$, which reads
    \begin{align}
        \la\phi_h-\phi_h^n,\psi_h\ra&= 0 \label{eq:G01}\\
        \la\mu_h,\xi_h\ra-\gamma\la\nabla\phi_h,\vtheta_h\nabla\xi_h\ra-\gamma\la\nabla\phi_h,\xi_h\nabla \vtheta_h\ra-\la f_{\phi}(\phi_h,\phi_h^n,\vtheta_h),\xi_h\ra &=0\label{eq:G02}\\
        \la s(\phi_h,\vtheta_h) - s(\phi_h^{n},\vtheta_h^{n}),\omega_h\ra&=0. \label{eq:G03}
	\end{align}
    
    First, we will show that \eqref{eq:G01}--\eqref{eq:G03} has a solution $(\phi_h,\mu_h,\vtheta_h)\in W$ for bounded $(\phi_h^n,\mu_h^n,\vtheta_h^n)$ with lower temperature bound $\underaccent{\bar}{\vtheta}^n.$ 
    By \eqref{eq:G01} $\phi_h=\phi_h^n$ and using the Taylor expansion we obtain from \eqref{eq:G03}
    \begin{equation*}
     \la \partial_\vtheta s(\phi_h^n,\zeta)(\vtheta_h-\vtheta_h^n),\omega_h\ra=0, \qquad \zeta\in[\vtheta_h^n,\vtheta_h].   
    \end{equation*}
    Recalling that $\partial_\vtheta s=-\partial_{\vtheta\vtheta} F$ and strict concavity of $F$, i.e. \ref{as:exergy}, we get that
    $\vtheta_h=\vtheta_h^n$. Then $\mu_h$ can be computed from \eqref{eq:G02}. It remains to show that $(\phi_h,\mu_h,\vtheta_h)\in W$. In fact it only remains to show the bound on $\norm{\mu_h}$. Hence, we estimate as follows:
    \begin{align*}
        \norm{\mu_h}_{L^2}^2 &= \gamma\la\vtheta_h\nabla\phi_h,\nabla\mu_h\ra+\gamma\la\nabla\phi_h\nabla \vtheta_h,\mu_h\ra+\la f_{\phi}(\phi_h,\phi_h,\vtheta_h),\mu_h\ra \\
        & \leq \gamma c_2\norm{\nabla\phi_h}_{L^2}\norm{\nabla\mu_h}_{L^2} + \gamma\norm{\nabla\phi_h}_{L^2}\norm{\nabla \vtheta_h}_{L ^2}\norm{\mu_h}_{L^\infty} \\
        &\;+ \norm{f_{\phi}(\phi_h,\phi_h,\vtheta_h)}_{L^1}\norm{\mu_h}_{L^\infty} \\
        &= (a) + (b) + (c).
    \end{align*}
    Using the inverse inequality \eqref{eq:inverse} we find that there exists constants $C_3,C_4,C_5>0$ such that
    \begin{align*}
        (a) &\leq  C_3h^{-1}\norm{\nabla\phi_h}_{L^2}\norm{\mu_h}_{L^2} , \\
        (b) &\leq  C_4h^{-(d+1)}\norm{\nabla\phi_h}_{L^2}\norm{\vtheta_h}_{L ^1}\norm{\mu_h}_{L^2},\\
        (c) &\leq C_5h^{-d/2}\norm{f_{\phi}(\phi_h,\phi_h,\vtheta_h)}_{L^1}\norm{\mu_h}_{L^2}.
    \end{align*}
    For (c) use \cref{as:fphi} as well as the Hölder and Sobolev inequalities to estimate
    \begin{align*}
        \norm{f_{\phi}(\phi_h,\phi_h,\vtheta_h)}_{L^1} &\leq C_f(\bar\vtheta^{n}+1)(2\norm{\phi_h}_{H^1}^6 + \snorm{\Omega}) \norm{\mu_h}_{L^2}.
    \end{align*}
    Since $\phi_h$ and $\vtheta_h$ are bounded we find
    \begin{equation}\label{eq:mubound}
        \norm{\mu_h}_{L^2} \leq C_\mu.
    \end{equation}
    
    Next, we will show that the topological degree is non-zero. To this end we compute the Jacobian matrix of $\G_0$
    \begin{equation*}
     J_{\G_0}(\x) = \begin{pmatrix}
         \mathbf{A} & \mathbf{0} & \mathbf{0} \\
         \mathbf{H}^\phi & \mathbf{A} & \mathbf{H}^\theta\\
          \mathbf{S}^\phi & \mathbf{0} & \mathbf{S}^\vtheta
     \end{pmatrix}  
    \end{equation*}
    with the matrices
    \begin{gather*}
        \mathbf{A}_{ij} = \la \lambda_j,\lambda_i \ra \quad\mathbf{S}^\phi_{i,j} = \la \dphi s(\phi_h,\vtheta_h)\lambda_j,\lambda_i \ra, \quad \mathbf{S}^\vtheta_{i,j} = \la \dtheta s(\phi_h,\vtheta_h)\lambda_j,\lambda_i \ra,\\
        \mathbf{H}_{i,j}^\phi = \gamma\la \nabla\lambda_j,\vtheta_h\nabla\lambda_i \ra +\gamma\la \nabla\lambda_j,\lambda_i\nabla\vtheta_h \ra + \la \partial_{\phi_h} f_\phi(\phi_h,\phi_h^n,\vtheta_h)\lambda_j,\lambda_i \ra,\\
        \mathbf{H}_{i,j}^\theta = \gamma\la \nabla\phi_h,\lambda_j\nabla\lambda_i\ra + \gamma\la \nabla\phi_h,\lambda_i\nabla\lambda_j \ra+ \la \partial_{\vtheta_h} f_\phi(\phi_h,\phi_h^n,\vtheta_h)\lambda_j,\lambda_i \ra.
    \end{gather*}
    Here the functions $\lambda_i$ are the finite element basis functions. The topological degree involves the signum of the determinant of the Jacobian matrix. We can compute that $\det(J_{\G_0}(\x))=\det(\mathbf{A})^2\det(\mathbf{S}^\vtheta).$ The mass matrix $\mathbf{A}$ is positive definite, hence $\det(\mathbf{A})^2> 0$. Since $\dtheta s = -\partial_{\vtheta\vtheta} F$, we find using strict concavity assumptions of $F$, i.e \ref{as:exergy}, that $-\partial_{\vtheta\vtheta}F> 0$ and therefore $\mathrm{sgn}(\det(J_{\G_0}(\x)))=1$. The topological degree is then given by
    \begin{equation*}
        \mathrm{deg}(\G_0,W,\mathbf{0}):=  \sum_{\x\in \G_0^{-1}(\mathbf{0})}\mathrm{sgn}\left(\det\left(J_{\G_0}(\x)\right)\right) = \sum_{\x\in \G_0^{-1}(\mathbf{0})} 1.
    \end{equation*}
    Hence, the topological degree is non-zero, if at least one $\x\in \G_0^{-1}(0)$ exists. We have already shown that system \eqref{eq:G01}--\eqref{eq:G03} admits at least one solution. Consequently, the topological degree is non-zero.
    
    Moreover, in this case the solution satisfies the assumptions of the implicit function theorem and hence there exists $\delta>0$ such that for all $\alpha\in[0,\delta)$ we have that $\mathbf{x}=z(\alpha)$, for a continuously differentiable function $z$. Due to the continuity of $\vtheta$ with respect to $\alpha$, the minimum of $\vtheta_h$ remains strictly positive for all $\alpha\in[0,\tfrac{\delta}{2}]$. We denote the minimal value of $\vtheta_h$ for $\alpha\in[0,\tfrac{\delta}{2}]$ by $C_\delta>0$.
    
    \noindent\textbf{Continuity of the map and condition (iii)}\\
    In this part we show that for every $\alpha\in[0,1]$, if a solution $\mathbf{x}\in V$ of $\mathbf{G}(\mathbf{x},\alpha)=0$ exists, then $\mathbf{x}\in W.$ To provide the required bounds of $\norm{\mathbf{x}},~\mathbf{x}=(\boldsymbol{\phi},\boldsymbol{\mu},\boldsymbol{\vtheta})$, we follow the lines of the proof of \cref{th:disspp} and find that for every $\alpha\in[0,1]$ we have
    \begin{align}
        \la\phi_h-\phi_h^n,1\ra&=0\label{eq:phimean}\\
        \la e(\phi_h,\vtheta_h)-e(\phi_h^n,\vtheta_h^n),1\ra &=\tau\mathcal{D}_{num}\leq0, \label{eq:ebound}\\
        \la s(\phi_h,\vtheta_h)-s(\phi_h^n,\vtheta_h^n),1\ra &=\tau\alpha\mathcal{D}_h\geq0.\label{eq:sbound}
    \end{align}
    Combination of the latter two yields
    \begin{align}
      \la e(\phi_h,\vtheta_h) - c_fs(\phi_h,\vtheta_h), 1\ra &- \tau\mathcal{D}_{num}  + c_f\tau\alpha\mathcal{D}_h \notag\\
      &=  \la e(\phi_h^n,\vtheta_h^n) - c_fs(\phi_h^n,\vtheta_h^n), 1\ra \leq C_{e,s}\label{eq:esbound}
    \end{align}
    for some constant $C_{e,s}=C_{e,s}(\norm{\phi_h^n}_{H^1},\norm{\vtheta_h^n}_{L^1},\underaccent{\bar}{\vtheta}^n)>0$, depending on the norm bounds for $\phi_h^n$ and $\vtheta_h^n$ and the lower bound $\underaccent{\bar}{\vtheta}^n$ for $\vtheta_h^n$. Then using \cref{as:exergy} we find
    \begin{align}
        \la  e(\phi_h,\vtheta_h) - c_fs(\phi_h,\vtheta_h), 1\ra &= \la b(\vtheta_h-c_f-c_f\log(\vtheta_h))\notag\\
        &\quad + f_2(\phi_h)-c_ff_1(\phi_h)+c_f\tfrac{\gamma}{2}\snorm{\nabla\phi_h}^2,1\ra.\label{eq:exbound}
    \end{align}
    Because $b(\vtheta_h-c_f-c_f\log(\vtheta_h))$ is bounded from below by $bc_f\log(c_f)$ for any $\vtheta_h$ and $f_2(\phi_h)-c_ff_1(\phi_h)>-c$, we obtain
    \begin{align*}
        \norm{\nabla\phi_h}_{L^2}^2 & %\leq \tfrac{2}{c_f\gamma}\left(C_{e,s}+\snorm{\Omega}\left(c_f\log(c_f)+c\right)\right)
        \leq 2\frac{C_{e,s}+\snorm{\Omega}\left(bc_f\log(c_f)+c\right)}{c_f\gamma}.
    \end{align*}
    Applying \eqref{eq:phimean} and the Poincaré inequality for $\phi_h$ we have
    \begin{equation}
      \norm{\phi_h}_{H^1}^2 \leq C_{\phi}.  \label{eq:phibound}
    \end{equation}
    To derive an upper bound for $\vtheta_h$ we observe that the internal energy has the form $e=\vtheta_h + f_2(\phi_h)$ with $f_2(\phi_h)>-c$. Using the positivity of $\vtheta_h$ and \eqref{eq:ebound} we find that $\norm{\vtheta_h}_{L^1}\leq C_{\norm{\vtheta}}$. Considering again \eqref{eq:exbound} yields 
    \begin{equation}
        \int_\Omega -\log(\vtheta_h)\leq \tfrac{c+C_{e,s}}{bc_f}+1=: C_{\log\vtheta}.	\label{eq:logbound}
    \end{equation}
    These bounds together with \eqref{eq:phibound} are sufficient to compute the uniform bound for $\mu_h$ in the same spirit of (ii), i.e. $\norm{\mu_h}_{L^2} \leq C_\mu$, for $C_\mu>0$. Hence, it remains to show the uniform lower bound of $\vtheta_h$.
    
    Recalling \eqref{eq:esbound} and the uniform positive definiteness assumption on the Onsager matrix $\L(\phi_h,\nabla\phi_h,\vtheta_h)$, i.e. \ref{as:onsag}, we find
    \begin{align*}
        c_f\tau\alpha\lambda_0\left(\norm*{\nabla \frac{1}{\vtheta_h}}_{L^2}^2 + \norm*{\nabla\frac{\mu_h}{\vtheta_h}}_{L^2}^2\right) \leq c_f\tau\alpha\mathcal{D}_h\leq C_{e,s}.
    \end{align*}
	We consider two different situations. First, we consider an element $K$ on which $\nabla\vtheta_h$ is non-zero and estimate
    \begin{align*}
        \frac{C_{e,s}}{c_f\tau\alpha\lambda_0}	\geq \int_K \snorm*{\nabla\frac{1}{\vtheta_h}}^2 = |\nabla\vtheta_h|^2\int_K \frac{1}{\vtheta_h^4} \geq C_{\nabla\vtheta}\norm{\vtheta_h^{-1}}_{L^4(K)}^4.
    \end{align*}
    As the gradient is by assumption non-zero we can restrict the integration domain from $K$ to $\tilde K:=\{x\in K: \vtheta_h(x)\in[\vtheta_{\min},2\vtheta_{\min}]\}$, where $\vtheta_{\min}$ is the minimum on $K$ and $|\tilde K|>0$. By construction we obtain
    \begin{equation*}
      \vtheta_{\min} \geq   \left(\frac{C_{\nabla\vtheta}c_f\tau\alpha\lambda_0|\tilde K|}{8C_{e,s}}\right)^{1/4} \geq \left(\frac{C_{\nabla\vtheta}c_f\tau\delta\lambda_0|\tilde K|}{16C_{e,s}}\right)^{1/4} := C_\text{grad}.
    \end{equation*}
    Second, we consider the case when $\vtheta_h=\vtheta_\text{const}$ is constant over one element $K$, i.e. $\nabla\vtheta_h\vert_K=0$. Then either $\vtheta_h = \vtheta_\text{const}$ over $\Omega$ or there exists an element $\bar K$ on which $\nabla\vtheta_h\neq 0$ and $\vtheta_h(x_j)=\vtheta_\text{const}$ at least in one nodal point $x_j$. If $\vtheta_h$ is constant in $\Omega$, we obtain using \eqref{eq:logbound} that 
    \begin{equation*}
        \vtheta_h \geq \exp\left(-\frac{C_{\log\vtheta}}{|\Omega|}\right).
    \end{equation*}
    In the other case, i.e. $\nabla\vtheta_h\neq 0$ on some element, we use the first result. Hence, we have derived the uniform lower bound 
    \begin{equation*}
        \vtheta_h \geq \min\left(\exp\left(-\frac{C_{\log\vtheta}}{|\Omega|}\right),C_\text{grad} \right)=:C_\text{min}
    \end{equation*}
    for $\alpha\in[\tfrac{\delta}{2},1]$. To obtain a uniform bound for $\alpha\in[0,1]$ we take the minimum of $C_\text{min}$ and $C_{\delta}$. We choose the constants $C_{\phi,\mu}, C_\vtheta, \varepsilon$ in \eqref{eq:defW} in the following way
    \begin{gather*}
        C_{\phi,\mu} = 2C_{\norm{\cdot}}\max\{\sqrt{C_\phi},C_\mu\}, \; C_\vtheta = 2\max\{C_{\norm{\cdot}_\infty}C_{\norm{\vtheta}},\bar\vtheta^{n}\}, \; \varepsilon=\frac{1}{2}\min\{C_\text{min},C_\delta\}.   %\label{eq:Wconstants}
    \end{gather*}
    Because $C_{\norm{\vtheta}}$ is the bound in the discrete $L^1$ norm, we multiply by the norm equivalence constant $C_{\norm{\cdot}_\infty}$. The same applies to the $L^2$ bound $C_{\phi,\mu}$ with the constant $C_{\norm{\cdot}}$.

    Finally, we can apply \cref{th:topo} for one time step of the problem $\mathbf{G}(\mathbf{x},\alpha)=0$ and find that at least one solution $\mathbf{x}\in W$ for $\mathbf{g}(\x)=0$ exists. The existence of at least one discrete solution $(\phi_h,\mu_h,\vtheta_h)$ can be done in an inductive way.
\end{proof}

\begin{remark}
    \begin{enumerate}
        \item We should note that \cref{th:existence} gives the existence of a bounded solution with strictly positive temperature. The bounds depend on $h,\tau$ due to the use of inverse inequalities. Note that the quasi-uniformity assumptions is only for simplicity and can be omitted, if local mesh sizes are considered. Furthermore, we expect the result to be extendable to a larger class of Helmholtz free energies $f$.
        \item Uniqueness of solution for small $\tau$ is a consequence of the implicit function theorem. However, for larger $\tau$ multiple solution may exist and can affect the convergence of the Newton method.
    \end{enumerate}
\end{remark}

\section{Numerical experiments}\label{sec:numerical}

In this section, we present numerical results for the NCH system \eqref{eq:sys1}, \eqref{eq:sys2}, \eqref{eq:sys3alternative}, using our numerical scheme \eqref{eq:dis1}-\eqref{eq:dis3}. We consider a convergence test as well as an example that illustrates the schemes behavior. We choose a domain $\Omega=[0,1]^2$ with periodic boundary conditions. We consider an example of a quenching process, where the solution starts in a mixed state and is then quenched below the critical temperature $\vtheta_c$ at the corners by setting a low initial temperature. At the rest of the domain $\Omega$, the temperature is above $\vtheta_c$. 
To model the temperature dependent separation of two phases, we introduce the driving potential
\begin{align}
    f&=a\big(2\phi^4-4\phi^3+(\tfrac{\vtheta-\vtheta_c}{d}+3)\phi^2
        -(\tfrac{\vtheta-\vtheta_c}{d}+1)\phi+0.25(\tfrac{\vtheta-\vtheta_c}{d}+1)\big)\notag\\
        &\quad-b(\vtheta \log(\tfrac{\vtheta}{\vtheta_c})+\vtheta-\vtheta_c)
\end{align}
with parameters $a=0.01,~b=1,~d=1$ and $\vtheta_c=3$. The graphs are depicted in \cref{fig:potential}. Mobility and heat conductivity are $\M=10^{-2}\cdot\I$ and $\mathbf{K} = 5\cdot10^{-3}\cdot\I$. The cross-coupling term $\C$ will be either $10^{-4}\cdot\I$ or $0\cdot\I$. The interface parameter is fixed to $\gamma=10^{-4}$. For simplicity we take $g^*$ in \eqref{eq:dis1}--\eqref{eq:dis3} are taken to be $g^n$. It means that they are evaluated explicitly and we set $\mu_h^0=0$. The resulting system is still nonlinear and solved using the Newton method with error tolerance $10^{-12}$. The resulting linear system is solved using a direct solver. The numerical method is implemented in NGSolve \cite{schoberl2014c++}.
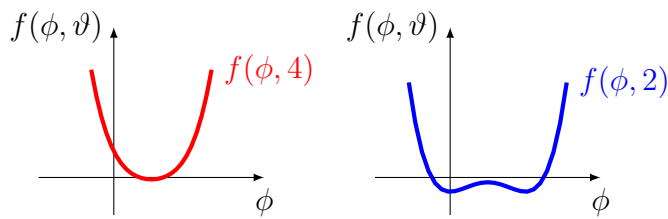
\begin{figure}[H]
    \centering
    \begin{tikzpicture}[scale=1,domain=-0.3:1.3]
        \draw[ -latex] (-1,0) -- (2,0) node[below] {$\phi$};
        \draw[ -latex] (0,-0.5) -- (0,2) node[left] {$f(\phi,\vtheta)$};
        %\draw[] (1,0.1) -- (1,-0.1) node[below] {1};
        %\draw[] (0.1,0.1) -- (-0.1,-0.1) node[below left] {0};
        \draw[red,ultra thick] plot ({\x},{2*(\x)^4-4*(\x)^3+4*(\x)^2-2*\x+0.5-4*ln(4/3)+1}) node[right]{$f(\phi,4)$};
    \end{tikzpicture}
    \begin{tikzpicture}[scale=1,domain=-0.55:1.55]
        \draw[ -latex] (-1,0) -- (2,0) node[below] {$\phi$};
        \draw[ -latex] (0,-0.5) -- (0,2) node[left] {$f(\phi,\vtheta)$};
        %\draw[] (1,0.1) -- (1,-0.1) node[below] {1};
        %\draw[] (0.1,0.1) -- (-0.1,-0.1) node[below left] {0};
        \draw[blue,ultra thick] plot ({\x},{2*(\x)^4-4*(\x)^3+2*(\x)^2-2*ln(2/3)-1}) node[right]{$f(\phi,2)$};
    \end{tikzpicture}
    \caption{The driving potential $f$ for different $\vtheta$}
    \label{fig:potential}
\end{figure}

\subsection{Convergence test}

We now present experimental results to verify the spatial and temporal convergence of the proposed scheme. We consider the following initial data
\begin{align*}
 \phi^0(x,y) &:= 0.6,\\
 \vtheta^0(x,y)&:= -5.9\cdot\operatorname{trans}_{[0,1]}\left(\frac{\sqrt{0.3x^2+0.3y^2}-0.2)}{\sqrt{0.001}}\right)+6.0,
\end{align*}
where
\begin{equation*}
    \operatorname{trans}_{[0,1]}(z)=\frac{-\tanh(z)+1}{2}\in[0,1].
\end{equation*}
% \begin{figure}[H]
%     \centering
%     %\vspace{-1.5cm}
%     \begin{tabular}{ccc}
%          \includegraphics[trim={17cm 0cm 17cm 0cm},clip,scale=0.165]{Bilder/Phi_CK_0.1_spacekonv_k0.png} 
%         &
%         \includegraphics[trim={17cm 0cm 17cm 0cm},clip,scale=0.165]{Bilder/Phi_CK_0.1_spacekonv_k4.png}
%         &
%         \multirow[t]{4}{*}[-6.85cm]{\includegraphics[trim={53cm 0cm 9.5cm 0cm},clip,scale=0.34]{Bilder/Phi_CK_0.1_timekonv_k0.png}}\\[-1em]
%         $h_k=2^{-0}$ & $h_k=2^{-4}$ &\\[-1em]
%         \includegraphics[trim={17cm 0cm 17cm 0cm},clip,scale=0.165]{Bilder/Phi_CK_0.1_timekonv_k0.png}
%         &
%         \includegraphics[trim={17cm 0cm 17cm 0cm},clip,scale=0.165]{Bilder/Phi_CK_0.1_timekonv_k8.png} &\\[-1em]
%         $\tau_k=2^{-0}\cdot10^{-1}$ & $\tau_k=2^{-8}\cdot10^{-1}$ &
%         %\multicolumn{2}{c}{\includegraphics[trim={14cm 1cm 10cm 50cm},clip,scale=0.335]{Bilder/Phi_C0_0.25.png}} \\[-0.5em]
%     \end{tabular}
%     \caption{Oben ist Raumkonvergenz und unten ist Zeitkonvergenz jeweils an Endzeit $T=0.1$}
%     \label{fig:bilderaaron}
%     \vspace{0.5cm}
% \end{figure}

As no exact solution is available, we will study self-convergence by computing the errors of the Cauchy sequence. We consider the following error norms: 
\begin{align*}
   \mathrm{err}(\nabla\phi,h_k)&:= \max_{n\in\Itau}\norm{\phi^n_{h_k}-\phi^n_{h_{k+1}}}_1 , \quad \mathrm{err}(\nabla\mu,h_k):=\sqrt{\tau\sum_{n=1}^{n_\tau}\norm{\mu^n_{h_k}-\mu^n_{h_{k+1}}}_1^2}, \\
   \mathrm{err}(\vtheta,h_k)&:=\max_{n\in\Itau}\norm{\vtheta^n_{h_k}-\vtheta^n_{h_{k+1}}}_0, \quad \mathrm{err}(\nabla\vtheta,h_k):=\sqrt{\tau\sum_{n=1}^{n_\tau}\norm{\vtheta^n_{h_k}-\vtheta^n_{h_{k+1}}}_1^2}.
\end{align*}
To compute the experimental order of convergence (eoc) in space we introduce $$\mathrm{eoc}_{k}:=\log_2\left(\frac{\mathrm{err}(a,h_{k-1})}{\mathrm{err}(a,h_{k})}\right)$$ for $a\in\{\nabla\phi,\nabla\mu,\vtheta,\nabla\vtheta\}.$

To investigate spatial convergence, we compute the error between successive spatial refinements on meshes with $h_k \approx 2^{-k}$ for $k=4,\ldots,8$ while keeping the time step fixed at $\tau=2^{-10}\cdot10^{-1}$. The final time is set to $T=10^{-1}$. The errors and approximate convergence rates, with $\C=10^{-4}\cdot\I$, can be found in \cref{tab:rates1}. They indicate optimal first-order convergence in the $H^1$-norm and the second-order convergence in the $L^2$-norm.

\begin{table}[htbp!]
    \centering
    %\small
    \caption{Errors and rates for convergence in space.\label{tab:rates1}} 
    \resizebox{\columnwidth}{!}{
    \begin{tabular}{c|l|c|l|c|l|c|l|c}
        % \hline
        %& \multicolumn{2}{c|}{$(\nabla\phi,h_k)$} & \multicolumn{2}{c|}{$(\nabla\mu,h_k)$} & \multicolumn{2}{c|}{$(\vtheta,h_k)$} & \multicolumn{2}{c}{$(\nabla\vtheta,h_k)$} \\
        %$ k $ &  \multicolumn{1}{c|}{err} &  eoc & \multicolumn{1}{c|}{err} & eoc & \multicolumn{1}{c|}{err} &  eoc & \multicolumn{1}{c|}{err} &  eoc   \\
        $ k $ & $\mathrm{err}(\nabla\phi,h_k)$ &  $\mathrm{eoc}_k$ & $\mathrm{err}(\nabla\mu,h_k)$ & $\mathrm{eoc}_k$ & $\mathrm{err}(\vtheta,h_k)$ &  $\mathrm{eoc}_k$ & $\mathrm{err}(\nabla\vtheta,h_k)$ &  $\mathrm{eoc}_k$   \\
        %$ k $ & $\mathrm{err}(\nabla\phi,h_k)$ &  $\mathrm{eoc}(\nabla\phi,h_k)$ & $\mathrm{err}(\nabla\mu,h_k)$ & $\mathrm{eoc}(\nabla\mu,h_k)$ & $\mathrm{err}(\vtheta,h_k)$ &  $\mathrm{eoc}(\vtheta,h_k)$ & $\mathrm{err}(\nabla\vtheta,h_k)$ &  $\mathrm{eoc}(\nabla\vtheta,h_k)$   \\
        \hline
        4 & $3.69\cdot 10^{-2}$ & -- & $9.15\cdot 10^{-3}$ & -- & $7.50\cdot 10^{-2}$ & -- & $1.97\cdot 10^{0}$ & --\\
        5 & $2.08\cdot 10^{-2}$ & $0.82$ & $5.63\cdot 10^{-3}$ & $0.70$ & $1.72\cdot 10^{-2}$ & $2.12$ & $9.92\cdot 10^{-1}$ & $0.99$\\
        6 & $1.12\cdot 10^{-2}$ & $0.90$ & $2.69\cdot 10^{-3}$ & $1.06$ & $4.36\cdot 10^{-3}$ & $1.98$ & $5.04\cdot 10^{-1}$ & $0.98$\\
        7 & $5.66\cdot 10^{-3}$ & $0.98$ & $1.27\cdot 10^{-3}$ & $1.08$ & $1.12\cdot 10^{-3}$ & $1.95$ & $2.50\cdot 10^{-1}$ & $1.01$\\
    \end{tabular}}
\end{table}
To investigate time convergence, we fix the mesh size at $h=2^{-7}$ and consider time step sizes $\tau_k=2^{-k}\cdot10^{-1}$ for $k=7,...,11$. The final time remains $T=10^{-1}$. The spatial mesh is fixed with mesh size $h=2^{-7}$. The temporal errors are defined analogously:
\begin{align*}
   \text{err}(\nabla\phi,\tau_k)&:= \max_{n\in\Itauk}\norm{\phi^n_{h,\tau_k}-\phi^n_{h,\tau_{k+1}}}_1 , ~\text{err}(\nabla\mu,\tau_k):=\sqrt{\tau_k\sum_{n=1}^{n_{\tau_k}}\norm{\mu^n_{h,\tau_k}-\mu^n_{h,\tau_{k+1}}}_1^2}, \\
   \text{err}(\vtheta,\tau_k)&:=\max_{n\in\Itauk}\norm{\vtheta^n_{h,\tau_k}-\vtheta^n_{h,\tau_{k+1}}}_0, ~\text{err}(\nabla\vtheta,\tau_k):=\sqrt{\tau_k\sum_{n=1}^{n_{\tau_k}}\norm{\vtheta^n_{h,\tau_k}-\vtheta^n_{h,\tau_{k+1}}}_1^2}.
\end{align*}
To compute the experimental order of convergence (eoc) in time we introduce $$\mathrm{eoc}_k:=\log_2\left(\frac{\mathrm{err}(a,\tau_{k-1})}{\mathrm{err}(a,\tau_{k})}\right)$$ for $a\in\{\nabla\phi,\nabla\mu,\vtheta,\nabla\vtheta\}$. The errors and rates, for $\C=10^{-4}\cdot\I$, can be found in \cref{tab:rates2}. As expected, optimal first-order convergence in time is achieved.
\begin{table}[htbp!]
    \centering
    \small
    \caption{Errors and rates for convergence in time.\label{tab:rates2}} 
    \resizebox{\columnwidth}{!}{
    \begin{tabular}{c|l|c|l|c|l|c|l|c}
        % \hline
        $ k $ & $\text{err}(\nabla\phi,\tau_k)$ &  $\mathrm{eoc}_k$ & $\text{err}(\nabla\mu,\tau_k)$ & $\mathrm{eoc}_k$ & $\text{err}(\vtheta,\tau_k)$ &  $\mathrm{eoc}_k$ & $\text{err}(\nabla\vtheta,\tau_k)$ &  $\mathrm{eoc}_k$   \\
        \hline
        7 & $4.56\cdot 10^{-4}$ & $0.86$ & $8.93\cdot 10^{-5}$ & $0.97$ & $2.11\cdot 10^{-4}$ & $1.00$ & $3.53\cdot 10^{-3}$ & $1.01$\\
        8 & $2.43\cdot 10^{-4}$ & $0.91$ & $4.52\cdot 10^{-5}$ & $0.98$ & $1.05\cdot 10^{-4}$ & $1.00$ & $1.76\cdot 10^{-3}$ & $1.00$\\
        9 & $1.26\cdot 10^{-4}$ & $0.94$ & $2.27\cdot 10^{-5}$ & $0.99$ & $5.25\cdot 10^{-5}$ & $1.00$ & $8.81\cdot 10^{-4}$ & $1.00$\\
        10 & $6.44\cdot 10^{-5}$ & $0.97$ & $1.14\cdot 10^{-5}$ & $0.99$ & $2.63\cdot 10^{-5}$ & $1.00$ & $4.40\cdot 10^{-4}$ & $1.00$\\
    \end{tabular}}
\end{table}
Similar convergence results in space and time can be obtained without the cross-coupling term, i.e. $\C=0\cdot\I$. 

\subsection{Illustrating example}

In this subsection we illustrate the behavior of the NCH system. We keep $\theta^0$ as before and choose
\begin{align*}
    \phi^0(x,y) &:= 0.5+0.01\sin(211\pi x)\sin(211\pi y).
\end{align*}
All parameters are choosen as above. We work with the mobility matrix $\M=\I$. Furthermore, we will compare the results obtained for $\C=0\cdot\I$ and $\C=10^{-4}\cdot\I$, i.e. with and without cross-coupling.

The temporal evolution of $\phi_h$ for $\C=0\cdot\I$ and $\C=10^{-4}\cdot\I$ is depicted in \cref{fig:phi}. The results of $\vtheta_h$ can be found in \cref{fig:theta}. 
\begin{figure}
    \centering
    %\vspace{-1.5cm}
    \begin{tabular}{c@{}c@{}c@{}c@{}}
        $t=0.1$ & $t=0.5$ & $t=2$ & $t=10$ \\[-0.5em]
         \includegraphics[trim={17cm 0cm 17cm 0cm},clip,scale=0.092]{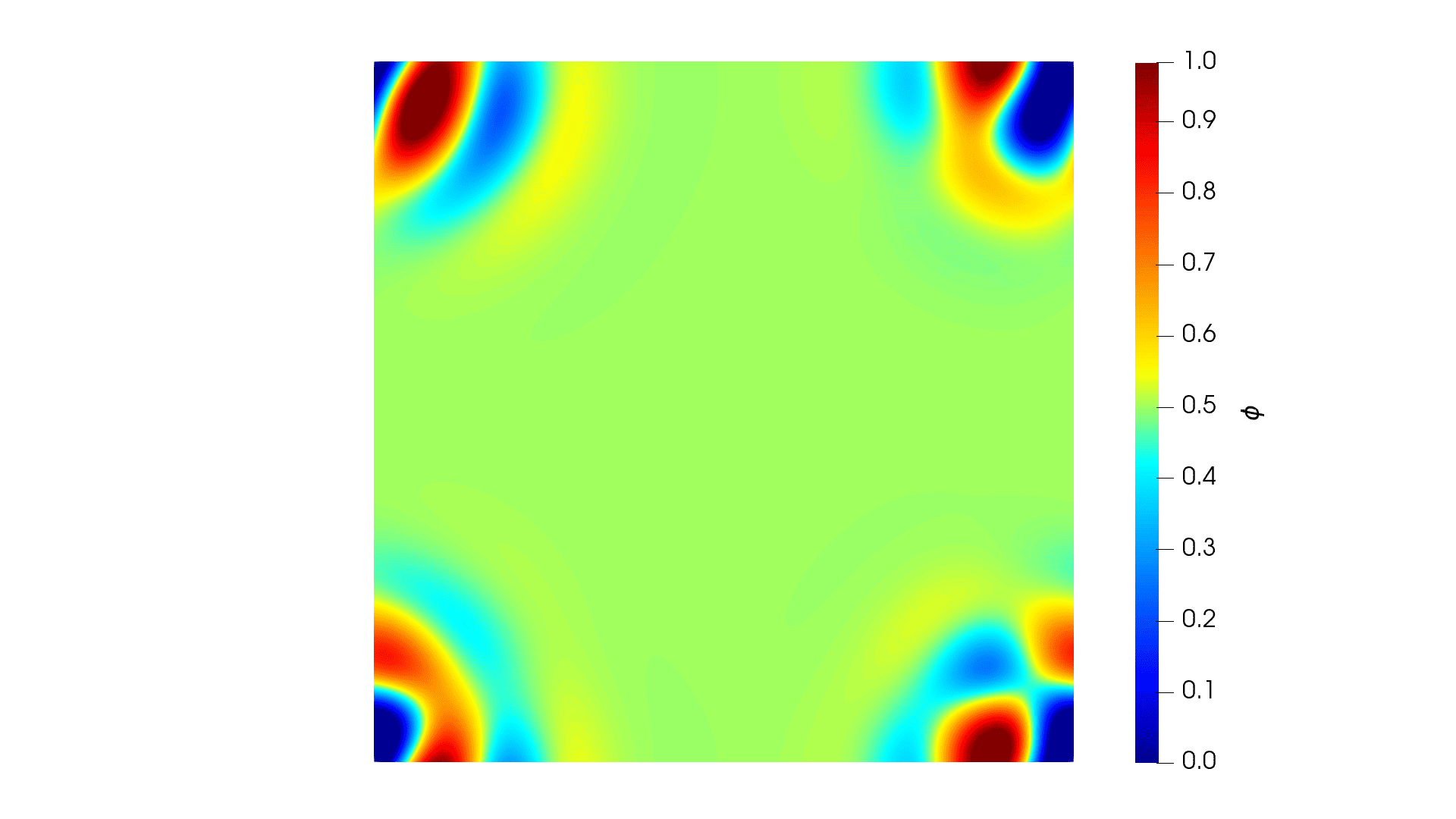} 
        &
        \includegraphics[trim={17cm 0cm 17cm 0cm},clip,scale=0.092]{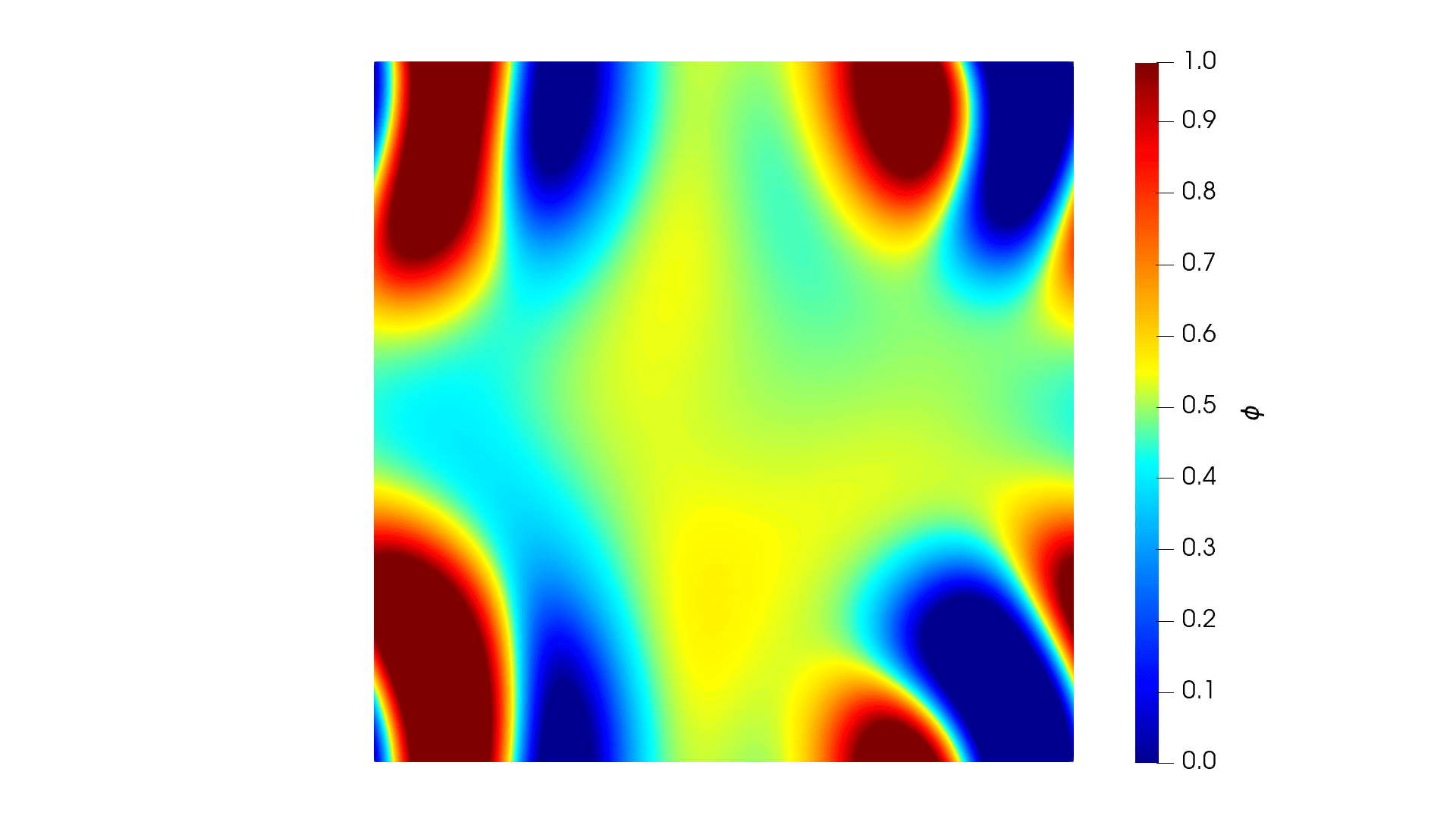} 
        &
        \includegraphics[trim={17cm 0cm 17cm 0cm},clip,scale=0.092]{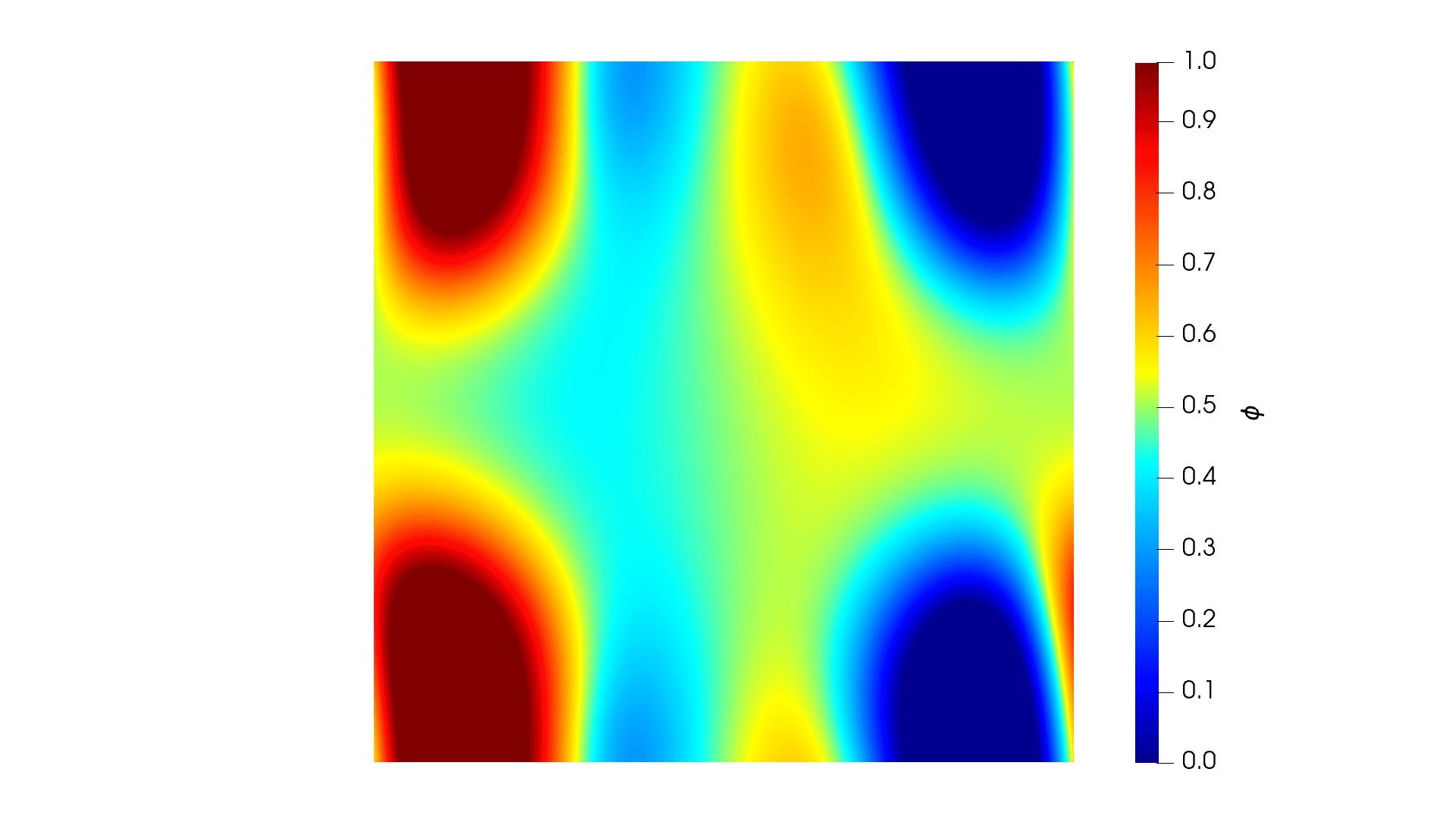}
        &
        \includegraphics[trim={17cm 0cm 17cm 0cm},clip,scale=0.092]{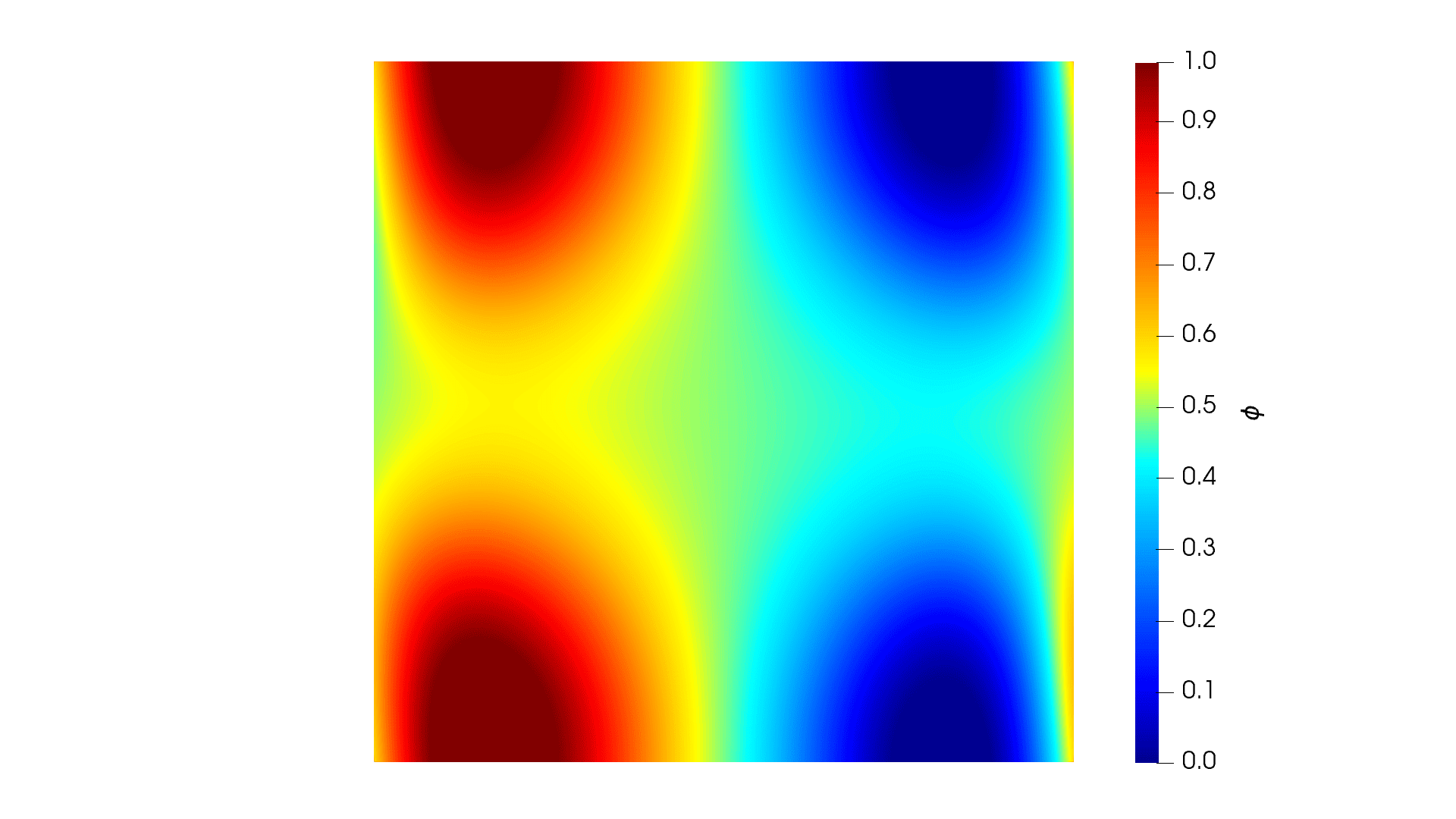} \\[-1em]
         \includegraphics[trim={17cm 0cm 17cm 0cm},clip,scale=0.092]{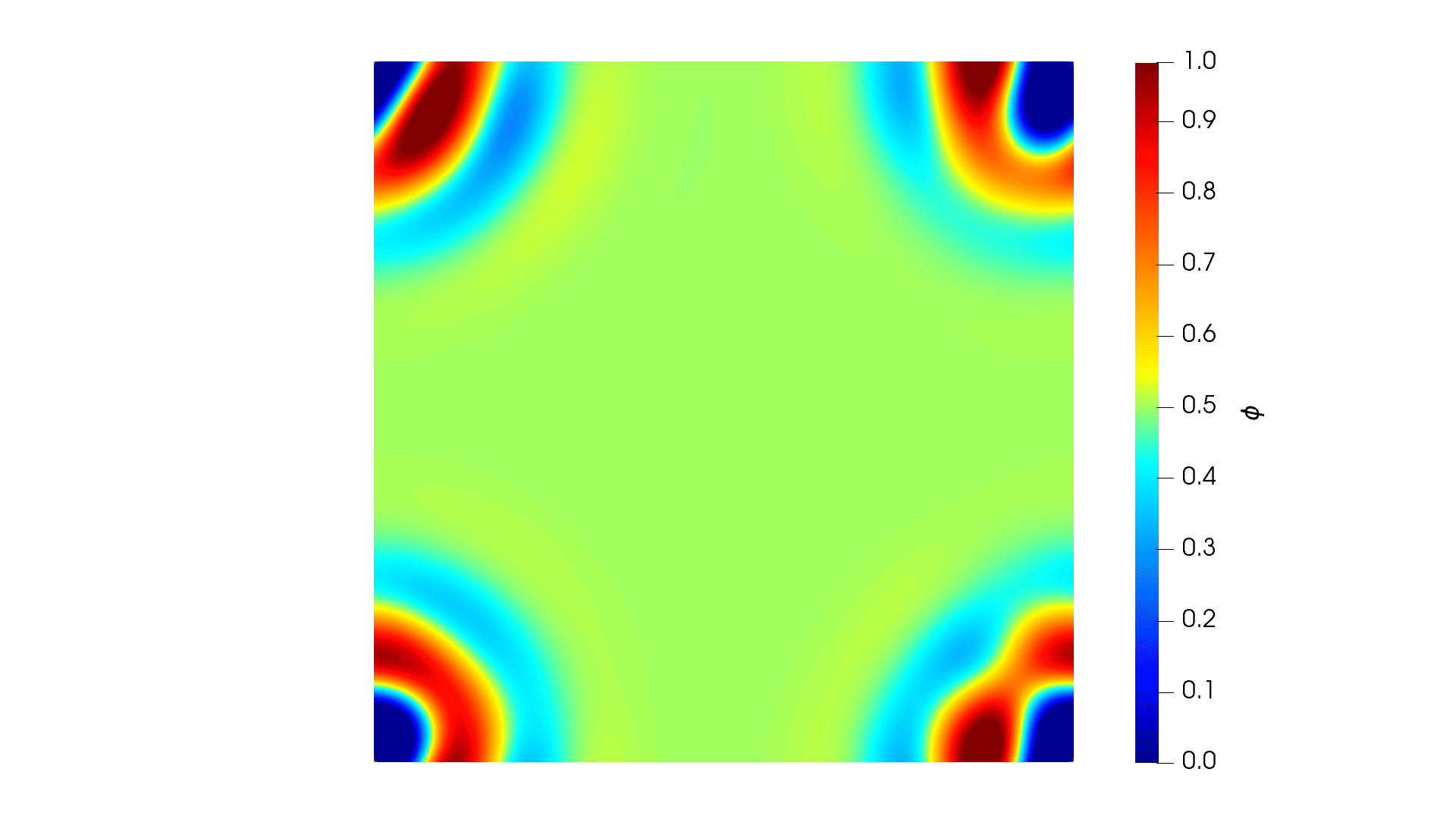} 
        &
        \includegraphics[trim={17cm 0cm 17cm 0cm},clip,scale=0.092]{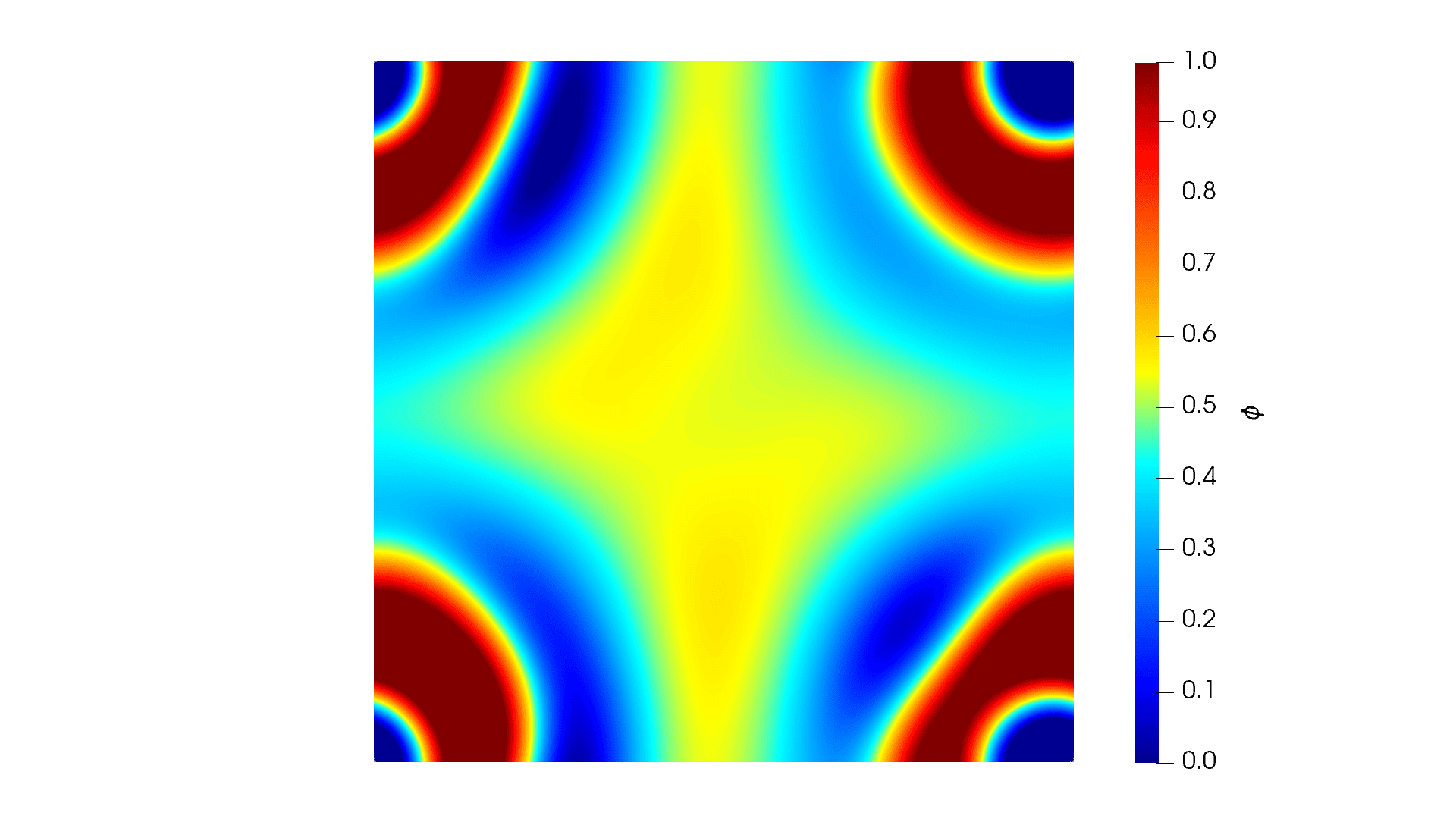}
        &
        \includegraphics[trim={17cm 0cm 17cm 0cm},clip,scale=0.092]{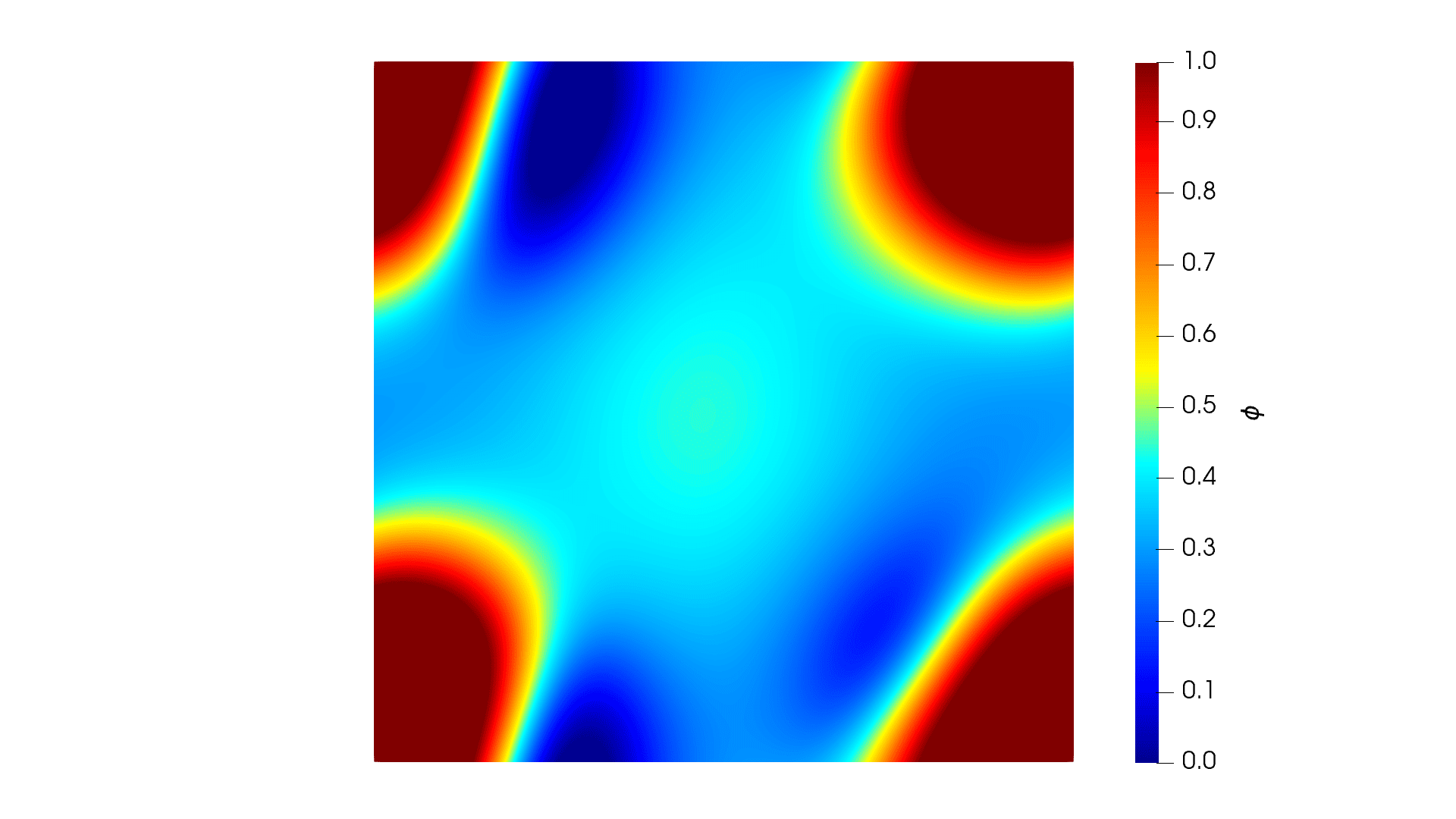}
        &
        \includegraphics[trim={17cm 0cm 17cm 0cm},clip,scale=0.092]{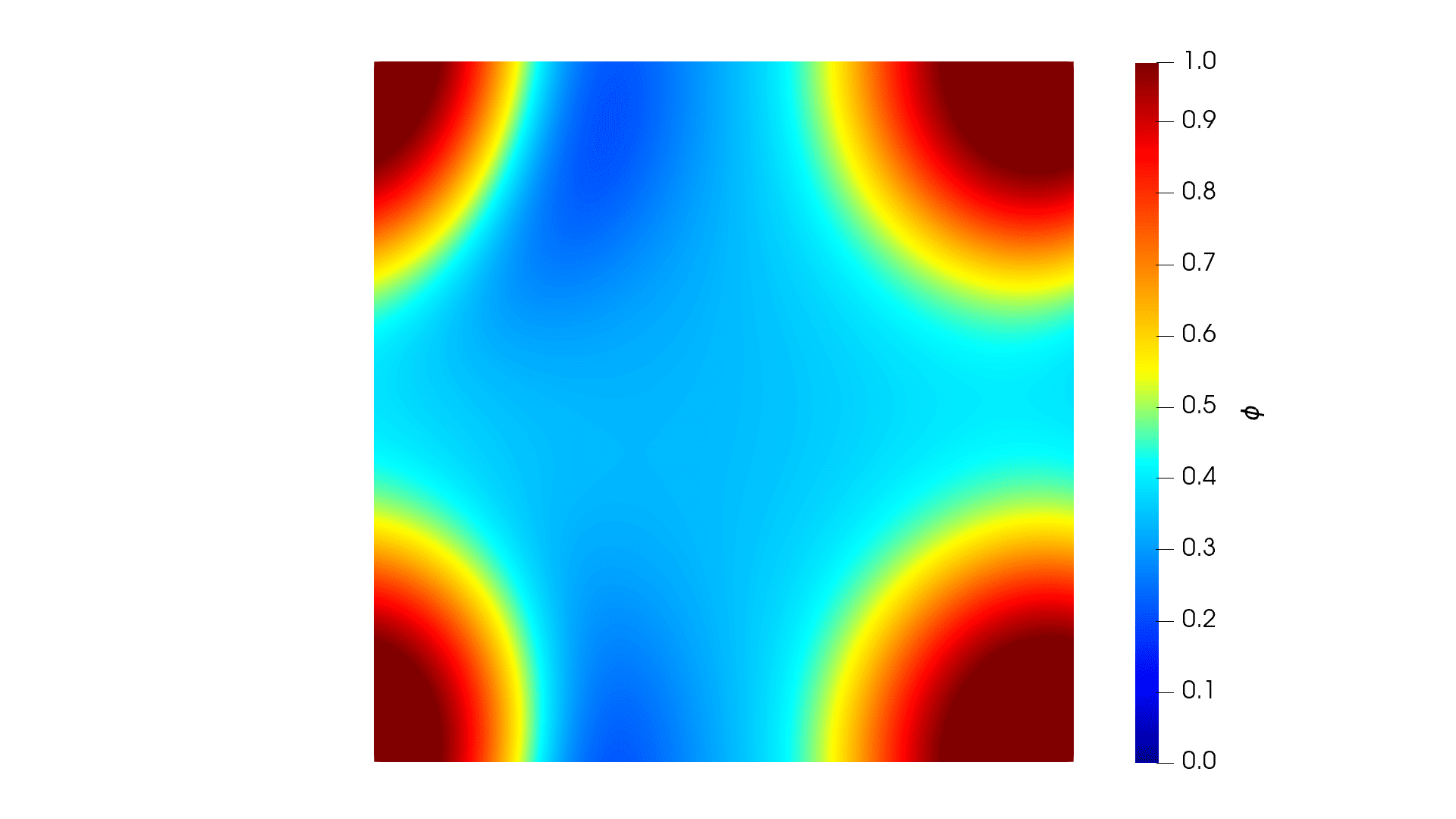} \\[-0.5em]
        \multicolumn{4}{c}{\includegraphics[trim={8.5cm 4.5cm 4.5cm 30cm},clip,scale=0.22]{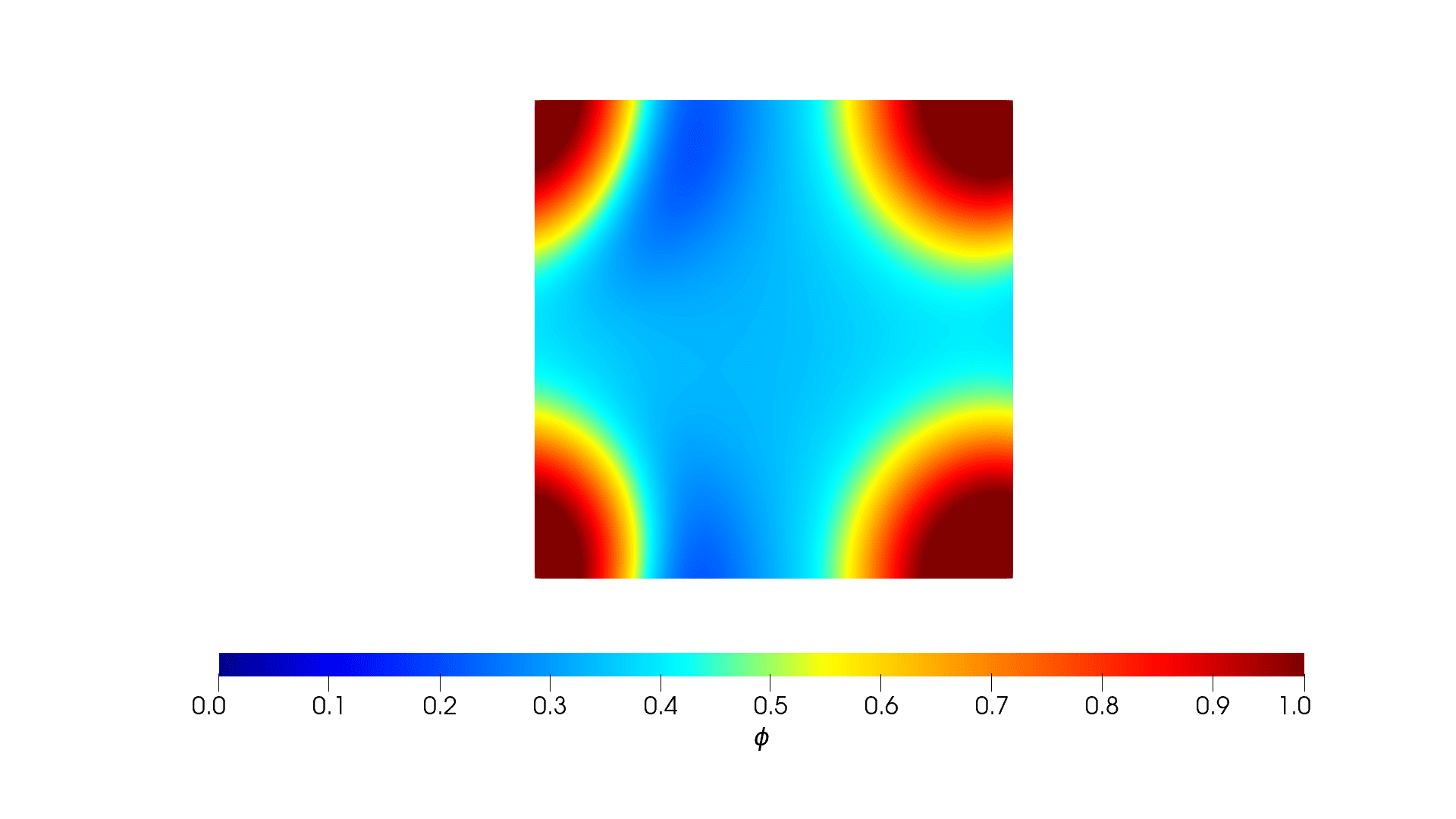}}
    \end{tabular}
    \caption{Snapshots of volume fraction $\phi$ for different cross-coupling matrices $\C$: (Top) $\C=0\cdot\I$; (Bottom) $\C=10^{-4}\cdot\I$.}
    \label{fig:phi}
    %\vspace{0.5cm}
\end{figure}
\begin{figure}
    \centering
    %\vspace{-1.5cm}
    \begin{tabular}{c@{}c@{}c@{}c@{}}
        $t=0.1$ & $t=0.5$ & $t=2$ & $t=10$ \\[-0.5em]
        \includegraphics[trim={17cm 0cm 17cm 0cm},clip,scale=0.092]{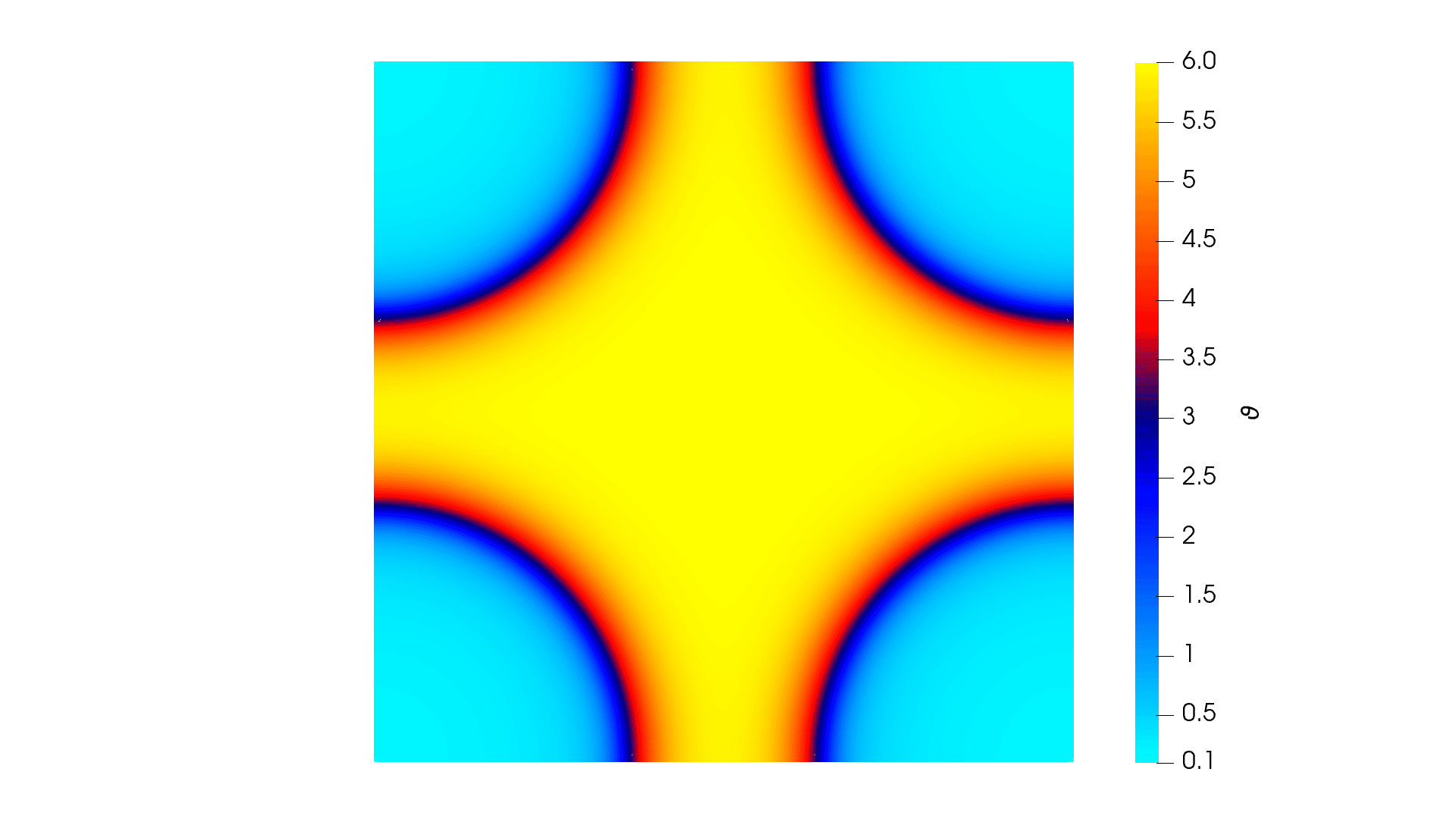} 
        &
        \includegraphics[trim={17cm 0cm 17cm 0cm},clip,scale=0.092]{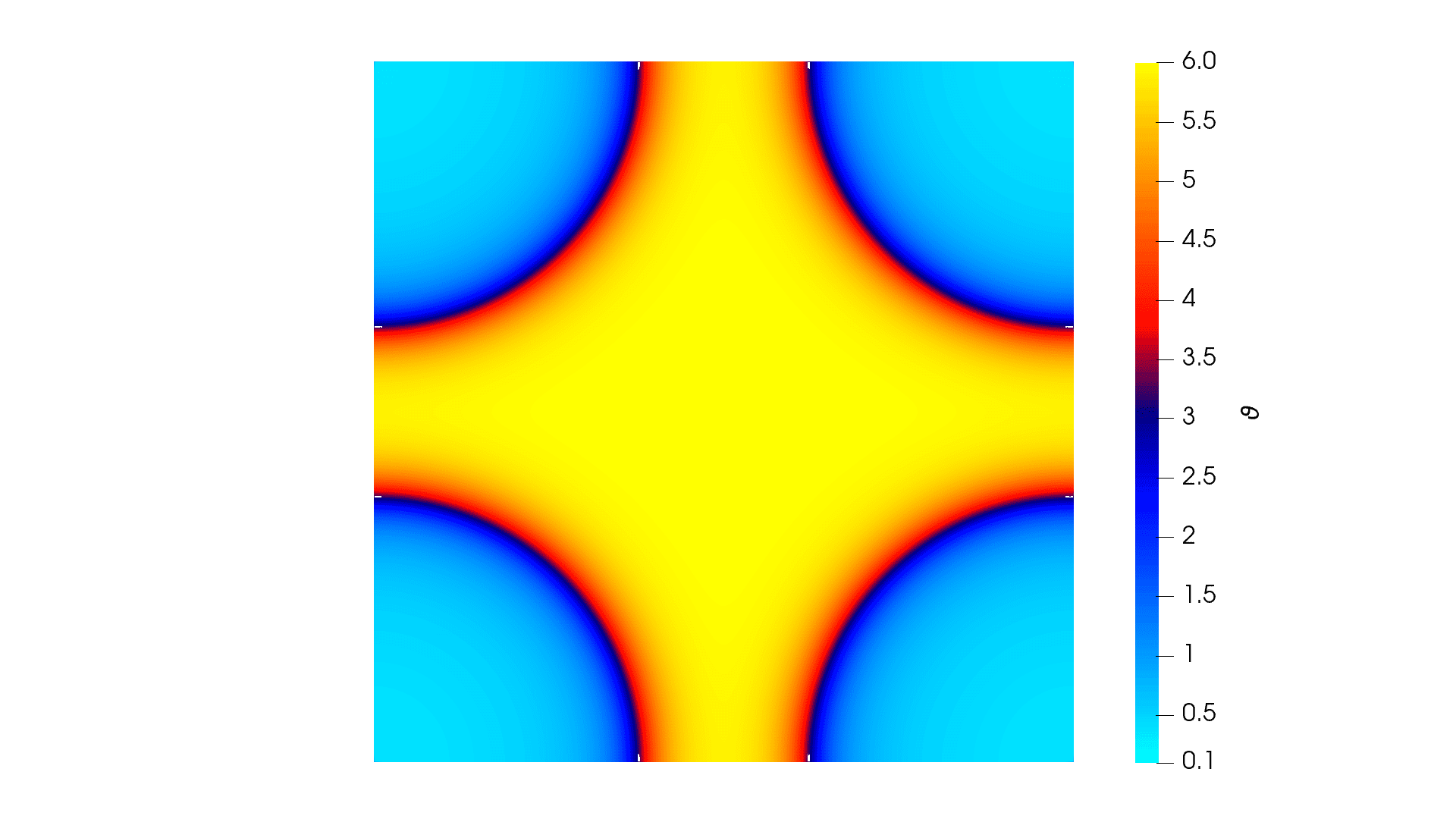}
        &
        \includegraphics[trim={17cm 0cm 17cm 0cm},clip,scale=0.092]{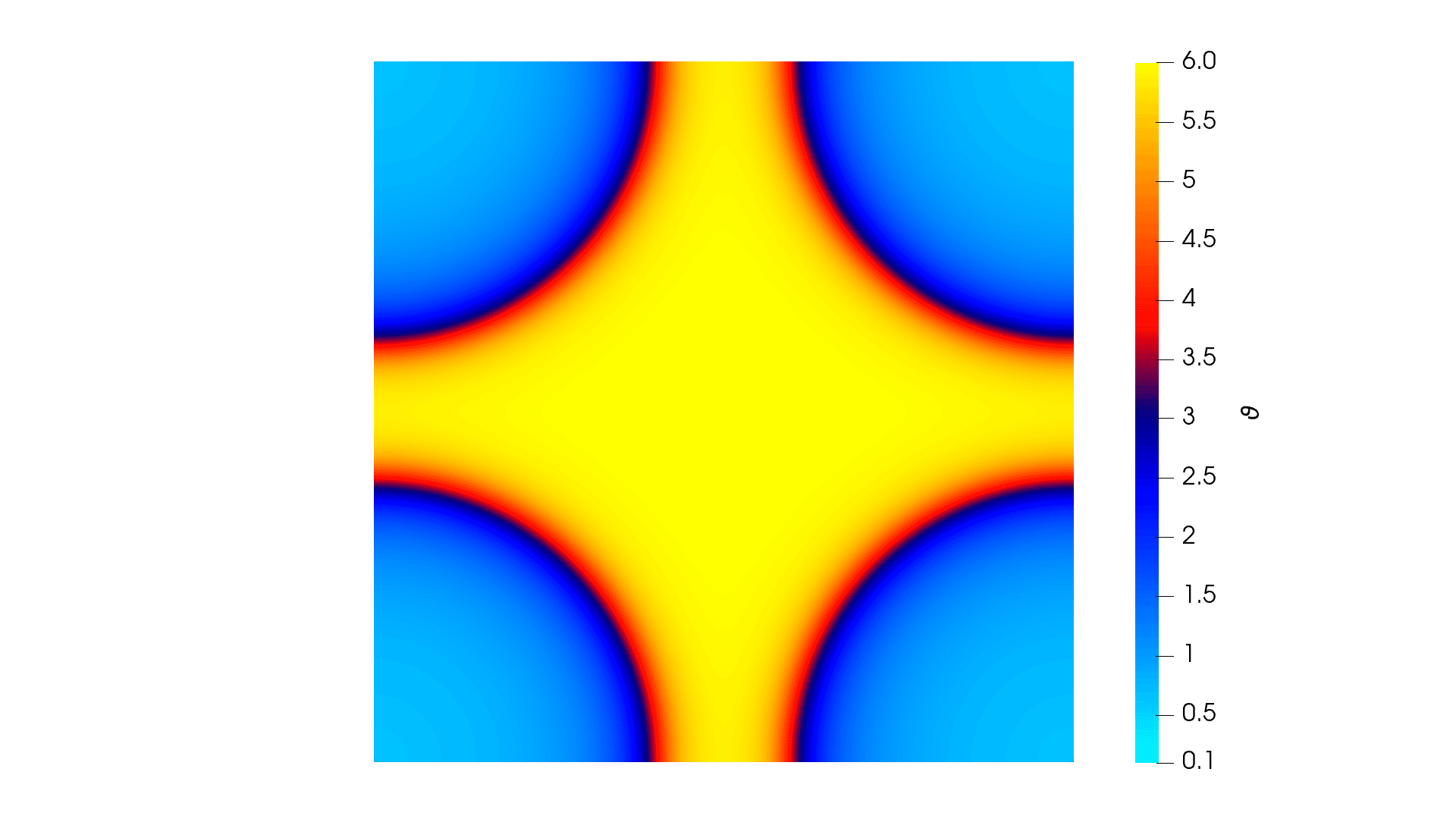}
        &
        \includegraphics[trim={17cm 0cm 17cm 0cm},clip,scale=0.092]{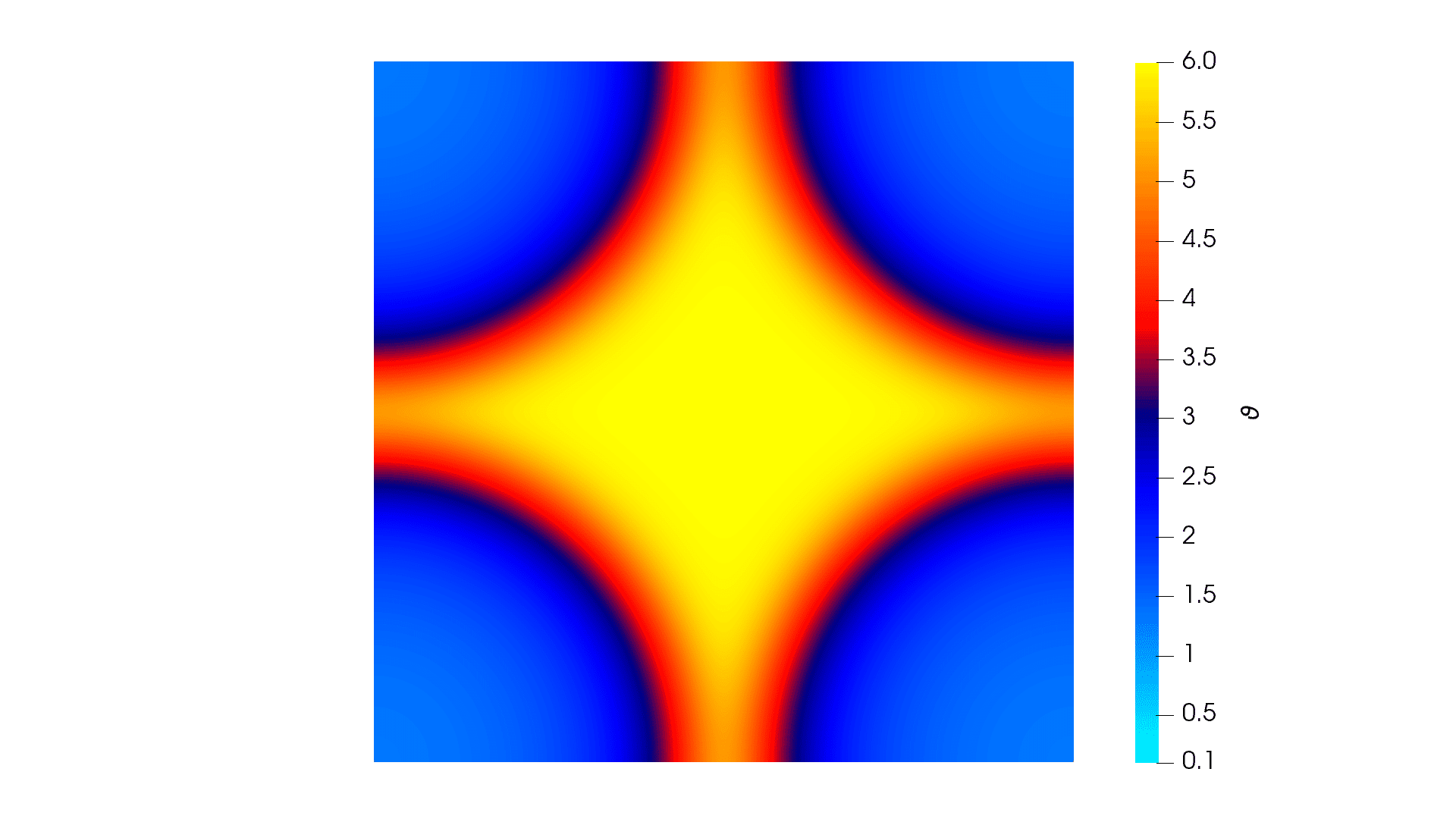} \\[-1em]
         \includegraphics[trim={17cm 0cm 17cm 0cm},clip,scale=0.092]{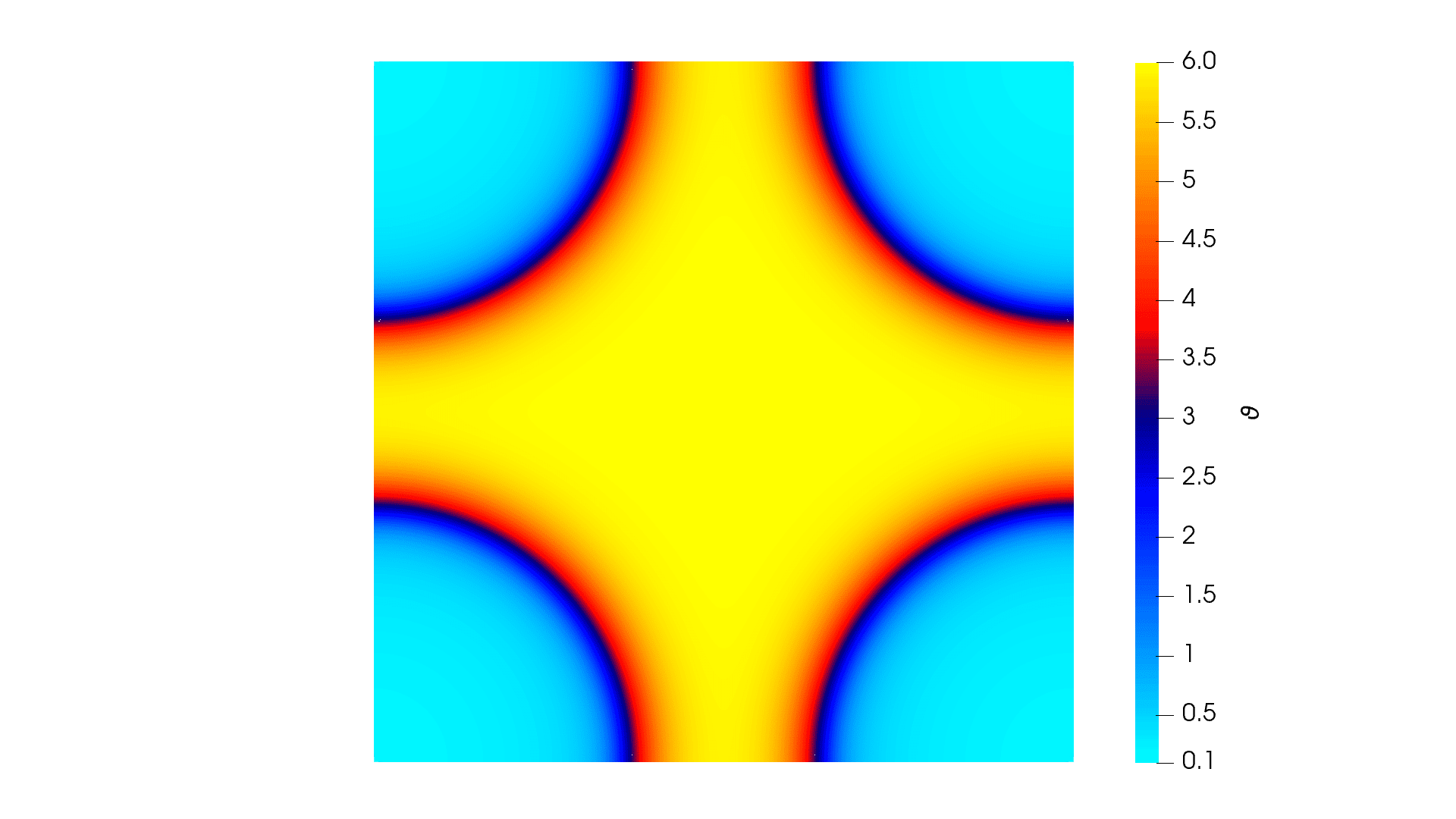}
        &
        \includegraphics[trim={17cm 0cm 17cm 0cm},clip,scale=0.092]{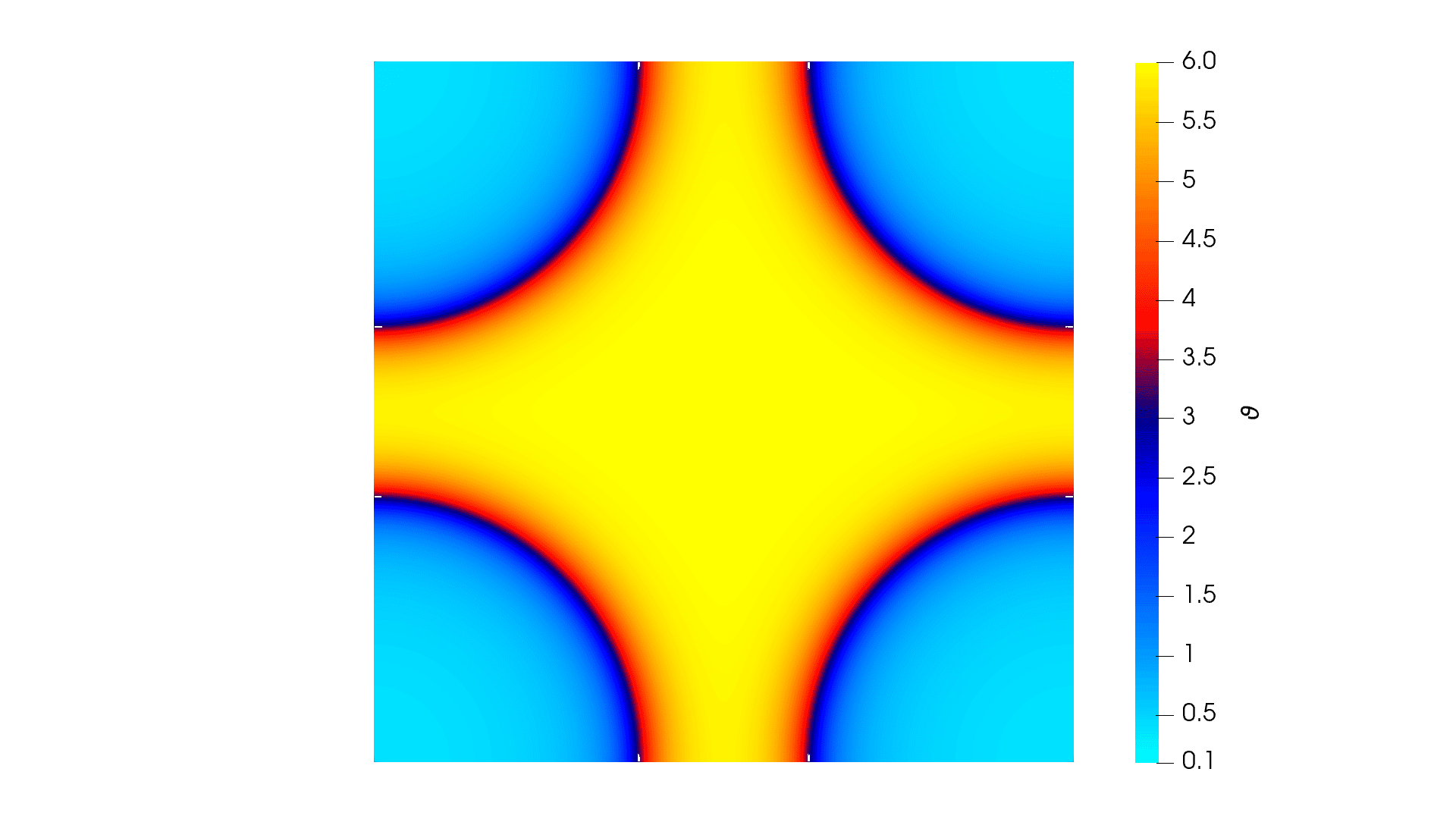}
        &
        \includegraphics[trim={17cm 0cm 17cm 0cm},clip,scale=0.092]{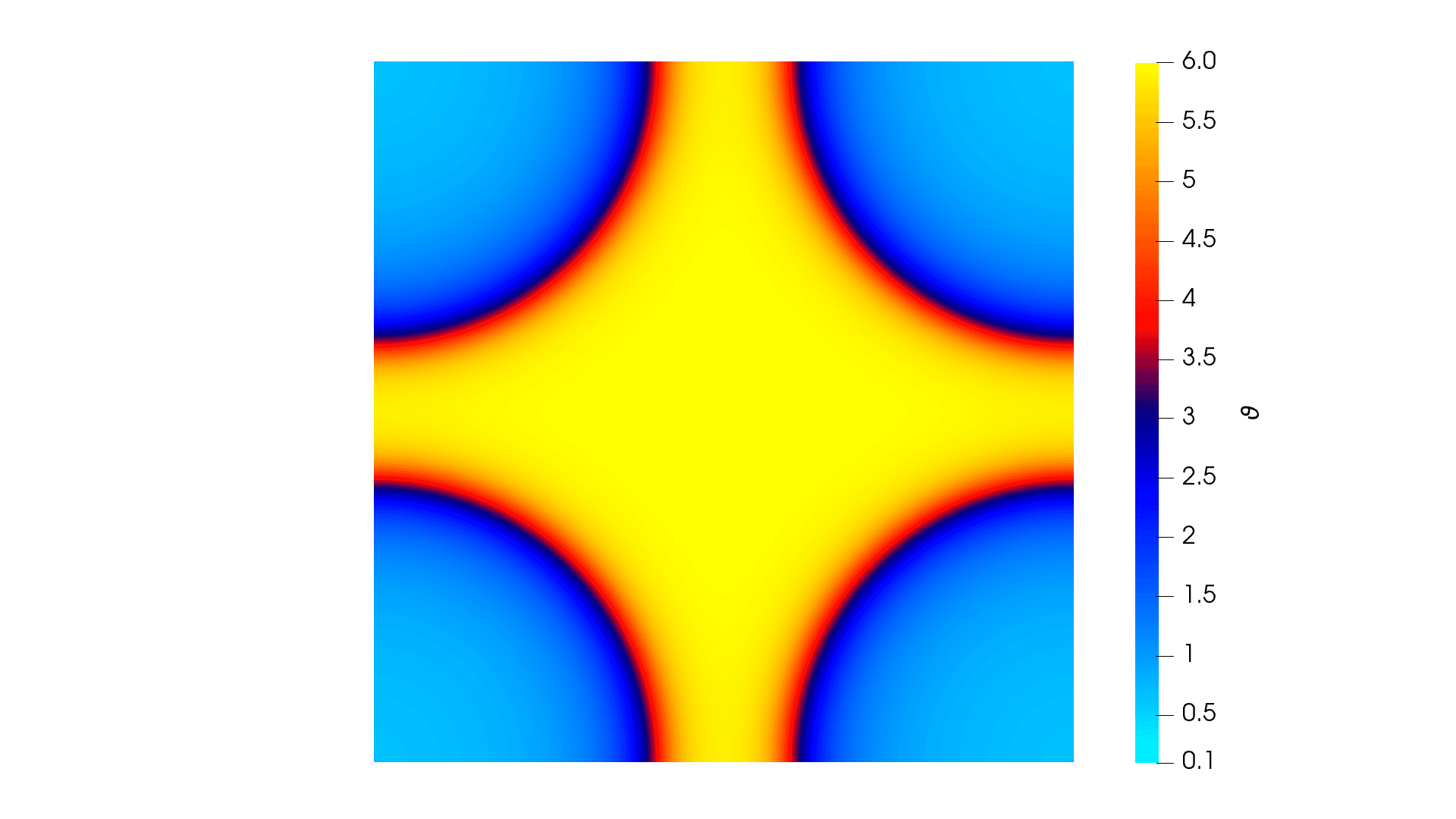}
        &
        \includegraphics[trim={17cm 0cm 17cm 0cm},clip,scale=0.092]{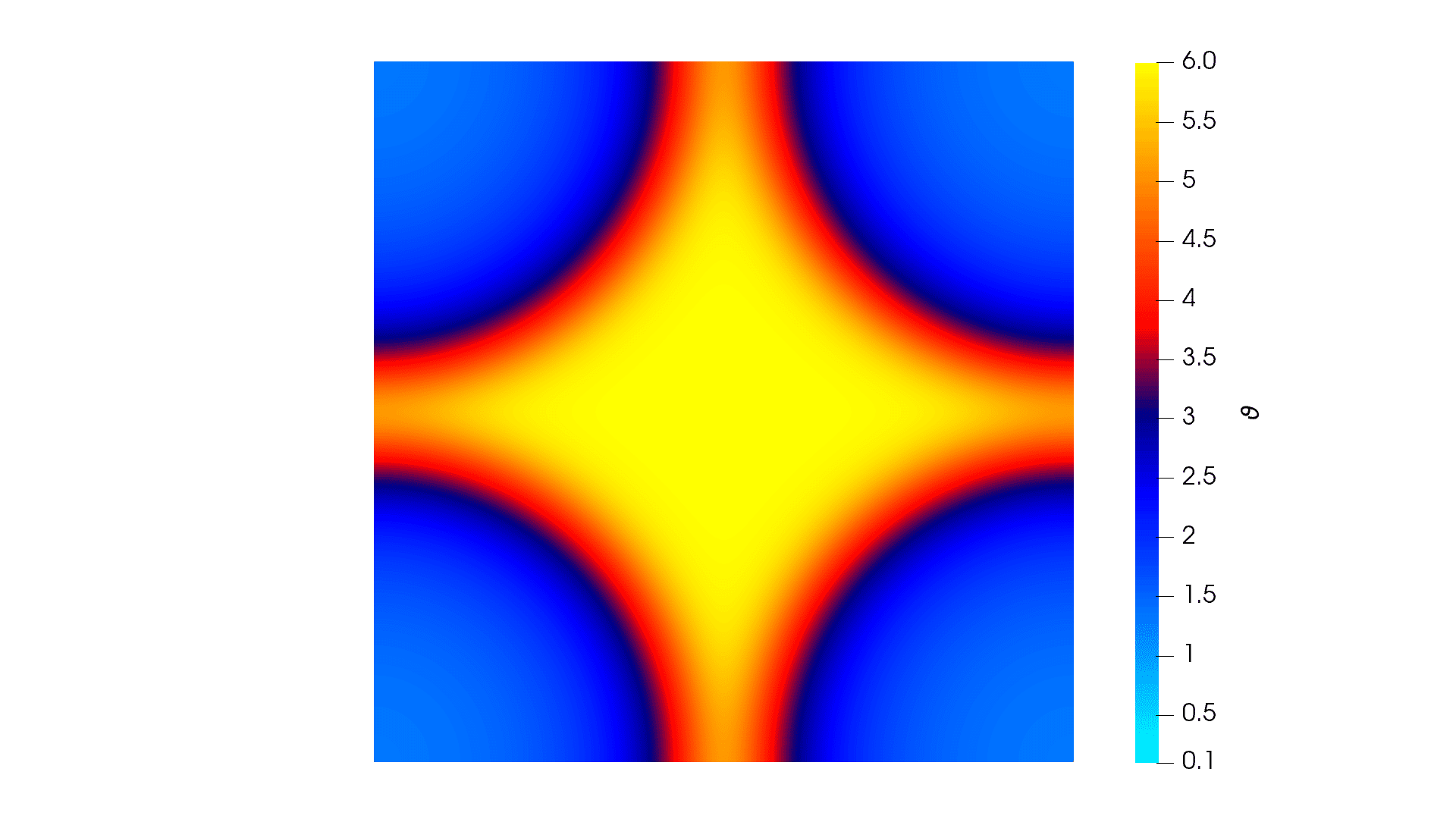} \\[-0.5em]
        \multicolumn{4}{c}{\includegraphics[trim={9cm 4.5cm 4.5cm 30cm},clip,scale=0.22]{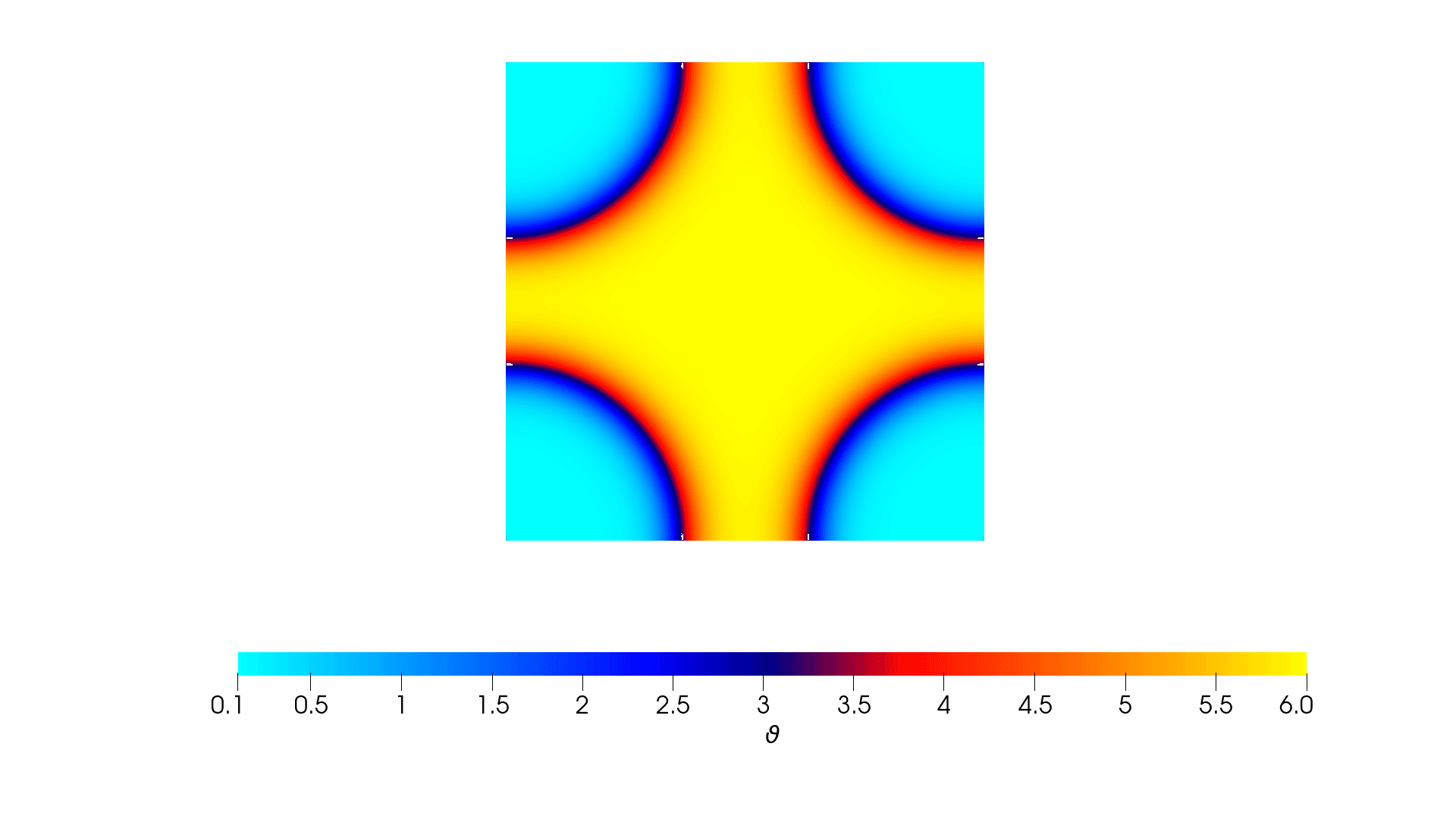}} \\[-0.5em]
    \end{tabular}
    \caption{Snapshots of temperature $\vtheta$ for different cross-coupling matrices $\C$: (Top) $\C=0\cdot\I$; (Bottom) $\C=10^{-4}\cdot\I$.}
    \label{fig:theta}
    %\vspace{0.5cm}
\end{figure}
We observe that in both cases the temperature profile is essentially the same. Due to the heat conduction, the sharp temperature gradients of the initial condition are smeared and heat dissipates into the colder regions. The temperature remains above the critical temperature $\vtheta_c$ in the middle of the domain and nearly zero at the corners, i.e. below the critical temperature. For the phase variable $\phi_h$, we see that for $\C=0\cdot\I$ a local phase separation occurs. This means that in a certain area around the four corners the mixture starts to separate into distinct phases. In the case of $\C=10^{-4}\cdot\I$, where we allow cross-coupling, the initial evolution is comparable, however we can observe more separation at the interface of two phases and also in the middle of the domain. This is done due to the additional driving force in the Cahn-Hilliard equation, leading to a final state with a fully connected one phase.
% \begin{figure}
%     \centering
%     \includegraphics[width=0.6\linewidth]{Bilder/Mass-difference_comp.png}
%     \caption{Mass conservation error for the cases $\C=0\cdot\I$ and $\C=10^{-4}\cdot\I$}
%     \label{fig:mass}
% \end{figure}
% \begin{figure}
%     \centering
%     \includegraphics[width=0.6\linewidth]{Bilder/Energy-difference_comp.png}
%     \caption{Energy conservation error for the cases $\C=0\cdot\I$ and $\C=10^{-4}\cdot\I$}
%     \label{fig:energy}
% \end{figure}
% \begin{figure}
%     \vspace{0.5cm}
%     \centering
%     \includegraphics[width=0.57\linewidth]{Bilder/Entropy_comp.png}
%     \caption{Temporal evolution entropy for the cases $\C=0\cdot\I$ and $\C=10^{-4}\cdot\I$}
%     \label{fig:entropy}
% \end{figure}
\begin{figure}[h]
     \centering
    %\vspace{-1.5cm}
    \begin{tabular}{c@{}c@{}}
        \includegraphics[width=0.49\linewidth]{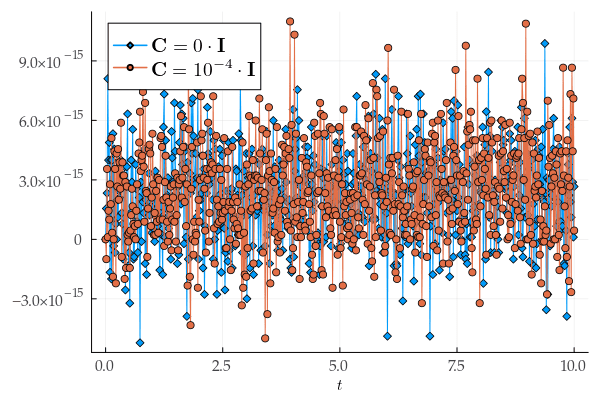} 
        &
        \includegraphics[width=0.49\linewidth]{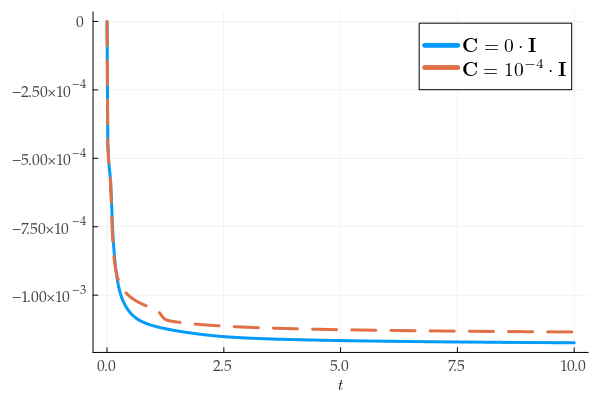}
        \\
        \multicolumn{2}{c}{\includegraphics[width=0.49\linewidth]{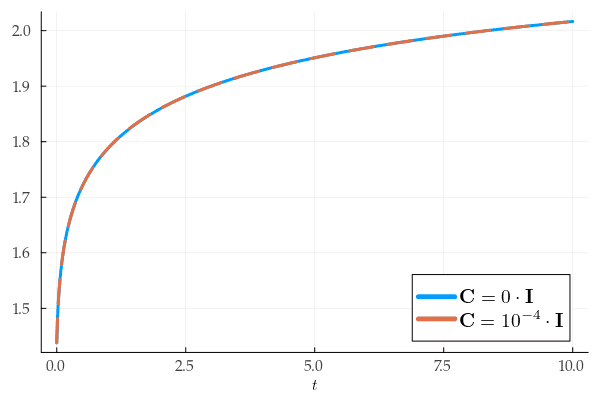}}   
    \end{tabular}
    \caption{Structure preserving properties over time (Top left): Mass conservation error; (Top right): Energy dissipation; (Bottom) Entropy production }
    \label{fig:struc}
\end{figure}

Structure-preserving properties, i.e. the conservation of mass, and energy dissipation as well as the entropy production, are shown in \cref{fig:struc} for both cases.
Numerical results confirm theoretical results from \cref{sec:structurepresurving}. Specifically, mass conservation up to a numerical error of $10^{-14}$, energy dissipation of order $10^{-3}$ as well as the entropy production. We note that the entropy production in both cases only differ by a factor of $10^{-4}$.

%We see in \cref{fig:mass} that, up to a numerical error of $10^{-14}$, the mass is conserved in both cases. Also in \cref{fig:energy} the energy dissipation of order $10^{-3}$ can be observed as well as the entropy production in \cref{fig:entropy}. We may note that the entropy in both cases only differ by a factors of $10^{-4}$. 

\section{Conclusion \& Outlook}

In this work, we proposed a fully discrete finite element method for the non-isothermal Cahn–Hilliard equation, based on a formulation using the entropy equation in place of the internal energy equation. The resulting scheme ensures mass conservation and entropy production, while the total energy—coinciding with the internal energy in this setting—is dissipated over time due to numerical diffusion. Our numerical experiments demonstrate optimal convergence rates: first-order in time and either first- or second-order in space, depending on the norm used. We also presented numerical results that highlight the influence of cross-coupling terms between the Cahn–Hilliard and entropy equations. Future work will aim at removing the artificial numerical dissipation to achieve discrete energy conservation. Additionally, we plan to extend the scheme to incorporate incompressible flow dynamics.

\section*{Acknowledgment}

The present research has been supported by the Deutsche Forschungsgemeinschaft (DFG, German Research Foundation) in the framework to the collaborative research center "Multiscale Simulation Methods for Soft-Matter Systems" (TRR 146) under Project No. 233630050 and by the SPP 2256 "Variational Methods for Predicting Complex Phenomena in Engineering Structures and Materials" under Project No. 441153493. M.L.-M. gratefully acknowledges the support of the Gutenberg Research Fellowship and the Mainz Institute
for Multiscale Modelling (M3odel).
\bibliographystyle{abbrv}
\bibliography{lit}
\end{document}